\documentclass[12pt]{article}
\usepackage{a4}
\usepackage{amsthm, amsmath, amsfonts, cite, enumitem}
\usepackage{graphicx}
\usepackage{caption}
\usepackage{titlesec}
\graphicspath{ {./Paperfigures/} }
\usepackage{epsfig}
\usepackage{caption}  
\usepackage{capt-of}
\usepackage{tikz}
\usepackage{multicol}
\usepackage{thmtools}
\usepackage{thm-restate}
\usepackage{cleveref}
\newtheorem{theorem}{Theorem}[section]
\newtheorem{proposition}[theorem]{Proposition}
\newtheorem{lemma}[theorem]{Lemma}

\newtheorem{conjecture}[theorem]{Conjecture}

\tikzset{every picture/.style={line width=0.75pt}} 
\tikzset{VertexStyle/.style = {shape= circle, ball color= white, inner sep= 0pt, outer sep= 0pt, minimum size = 30 pt, }}  

\titleformat{\section}
{\normalfont\fontsize{15.5}{18}\bfseries}{\thesection}{1em}{}

\title{Edge-maximal graphs on orientable and some non-orientable surfaces}
\author{James Davies\thanks{Department of Combinatorics and Optimization, University of Waterloo, Waterloo, Canada. E-mail: \texttt{jgdavies@uwaterloo.ca}.} \ and Florian Pfender\thanks{Department of Mathematical and Statistical Sciences, University of Colorado Denver, Denver, USA. Email: \texttt{florian.pfender@ucdenver.edu}. Research of this author is partially supported by NSF grant DMS-1855622.}}
\date{}

\begin{document}

\maketitle
	
\begin{abstract}
	We study edge-maximal, non-complete graphs on surfaces that do not triangulate the surface.
	
	We prove that there is no such graph on the projective plane $\mathbb{N}_1$, $K_7-e$ is the unique such graph on the Klein bottle $\mathbb{N}_2$ and $K_8-E(C_5)$ is the unique such graph on the torus $\mathbb{S}_1$.
	
	In contrast to this for each $g\ge 2$ we construct an infinite family of such graphs on the orientable surface $\mathbb{S}_g$ of genus $g$, that are $\lfloor \frac{g}{2} \rfloor$ edges short of a triangulation.
\end{abstract}

\section{Introduction}

All graphs considered in this paper are simple. A graph $G$ is edge-maximal in a graph class $\mathcal{G}$ if $G\in \mathcal{G}$, and the addition of any missing edge $e\notin E(G)$ yields a graph $G+e\notin\mathcal{G}$. In this paper, the graph classes we are interested in are the graphs embeddable in a given surface, and we try to understand the edge maximal graphs in these classes.

It is well known and straightforward to show that every edge-maximal planar graph either triangulates the surface or is a complete graph on at most two vertices. While every graph that triangulates a given surface, and every embeddable complete graph is trivially edge-maximal, the reverse of this statement is not true in general for other surfaces. As first observed by Franklin~\cite{franklin1934six}, $K_7-e$ is edge-maximally embeddable on the Klein bottle, while it is one edge short of the number of edges required for a triangulation given by Euler's formula. However, for some low genus surfaces we prove that there are few exceptions. This answers a question of McDiarmid and Wood \cite{mcdiarmid2018edge}.

\begin{theorem}\label{projectivePlane}
	Every edge-maximal projective-planar graph either triangulates the surface or is complete.
\end{theorem}

\begin{theorem}\label{klein}
	With the exception of $K_7-e$, every edge-maximal graph embeddable on the Klein bottle either triangulates the surface or is complete.
\end{theorem}

\begin{theorem}\label{torus}
	With the exception of $K_8-E(C_5)$, every edge-maximal graph embeddable on the torus either triangulates the surface or is complete.
\end{theorem}
	
In 1972, Harary, Kainen, Schwenk and White conjectured that there is an edge-maximal,  non-complete and non-triangulating graph on every orientable surface other than the sphere \cite{harary1974maximal}. We prove this conjecture in a rather strong sense.

\begin{theorem}\label{1972 impures}
	Let $\mathbb{S}_g$ be the orientable surface of genus $g\ge 2$. Then there exists infinitely many edge-maximal graphs on $\mathbb{S}_g$ that are $\lfloor \frac{g}{2} \rfloor$ edges short of a triangulation.
\end{theorem}

In 1974, Kainen asked how many edges short of a triangulation can a graph on a given surface $\Sigma$ be \cite{kainen1974some}? McDiarmid and Wood proved an upper bound linear in the Euler genus of the surface \cite{mcdiarmid2018edge}. Together with Theorem~\ref{1972 impures}, this  asymptotically answers Kainen's question for orientable surfaces up to a multiplicative constant.

Additionally, McDiarmid and Wood asked if $G$ is embeddable in a surface $\Sigma$, and has sufficiently many vertices but is not edge-maximal, can one always add edges to obtain a triangulation of $\Sigma$ \cite{mcdiarmid2018edge}? We give a strong negative answer to this question for all orientable surfaces of genus $g\ge 2$.

\begin{theorem}\label{Wood graphs}
	Let $\mathbb{S}_g$ be an orientable surface of genus $g\ge 2$ and fix $n\ge 1$. Then there exist infinitely many graphs on $\mathbb{S}_g$ which are at least $n$ edges short of triangulating $\mathbb{S}_g$ but no super-graph on the same vertex set is less than $\lfloor \frac{g}{2} \rfloor$ edges short of triangulating $\mathbb{S}_g$.
\end{theorem}

A graph $G$ embeddable on a surface $\Sigma$ is \emph{Euler impure} if it is non-complete, edge-maximal and does not triangulate the surface.

Previously, not many Euler impure graphs on surfaces were known.
As we mentioned above, Franklin gave a first example with $K_7-e$ on the Klein bottle $\mathbb{N}_2$ \cite{franklin1934six}.
Ringel proved that $K_8-E(2K_2)$ and $K_8-E(K_{1,2})$ are both Euler impure on Dyck's surface $\mathbb{N}_3$ \cite{ringel1955man}.
Huneke and independently Jungerman and Ringel proved that all 9-vertex, 32-edge graphs embeddable on the double torus $\mathbb{S}_2$ are Euler impure \cite{huneke1978minimum,jungerman1980minimal}.
In almost all cases, Euler's formula permitting, there exists a simple $n$-vertex graph triangulating a given surface $\Sigma$. The preceding Euler impure graphs arise from the three exceptional cases where this does not hold:  $n=7$ on $\mathbb{N}_2$, $n=8$ on $\mathbb{N}_3$, and $n=9$ on $\mathbb{S}_2$ \cite{jungerman1980minimal,ringel1955man}.

Just one other Euler impure graph on a surface was known. Harary, Kainen, Schwenk and White proved that $C_3+C_5$ (the graph obtained from $K_8$ by removing the edges of a $C_5$) is Euler impure on the torus $\mathbb{S}_1$ \cite{harary1974maximal,duke1972genus}.

Many proofs on planar graphs are simplified by reducing to the case of having a triangulation as this gives some simple local structure. The existence of Euler impure graphs on surfaces means that this technique is not generally possible for surfaces other than $\mathbb{S}_0$. After possibly dealing with exceptional cases, Theorems~\ref{projectivePlane},~\ref{klein}, and~\ref{torus} allow for this technique on the projective plane, Klein bottle and the torus, respectively. The first author recently made use of Theorem \ref{projectivePlane} in this way with a discharging argument. They proved that if $G$ is a projective-planar graph with minimum degree 5 and $A_k$ is the set of vertices of degree $k$, then there exists a component of $G[A_5\cup A_6]$ containing at least three vertices of $A_5$ \cite{Davies2019}.

Similarly to Euler impure graphs, given a class of graphs $\cal G$, an edge maximal graph $G\in {\cal G}$ is \emph{$k$-impure} if $\max_{H\in {\cal G}}\{E(H)-E(G): V(H)=V(G)\}\ge k$. With the exception of the graphs arising when there is no simple $n$-vertex graph triangulation of the surface, the notion of impure and Euler impure are equivalent. We make this distinction as we feel that these exceptional graphs that are Euler impure but not impure provide important examples for studying impure graphs on surfaces.

The impurity of a graph class can be thought of as a saturation problem. Graph purity has been studied in other contexts. McDiarmid and Przykucki determined for which graphs $H$, there are no impure graphs in the class of $H$-minor free graphs. Additionally they proved that $H$-minor free graphs are either at most $k$-impure for some $k$, or there exists graphs that are $\Omega(|V|)$-impure \cite{mcdiarmid2019purity}, where $|V|$ is the number of vertices in the graph. Dehkordi and Farr recently constructed an infinite family of edge-maximal linklessly embeddable graphs with $3|V|-3$ edges \cite{dehkordi2019non}, and so ($|V|- 7$)-impure \cite{sachs1983spatial,mader1968homomorphiesatze}.

This paper is organized as follows. Section 2 is dedicated to proving Theorems \ref{projectivePlane}, \ref{klein} and \ref{torus}. In Section 3 we construct the graphs in Theorems \ref{1972 impures} and \ref{Wood graphs}. Additionally we build machinery for constructing more similar graphs. Lastly in Section 4 we discuss possible directions for further work and make a conjecture on the Euler impure graphs on Dyck's surface $\mathbb{N}_3$ analogous to Theorems \ref{projectivePlane}, \ref{klein} and \ref{torus}.

\section{Edge-maximal graphs on low genus surfaces}

Given an embedding of an Euler impure graph $G$ on a surface $\Sigma$, we can consider a non-triangular face $F$. Let $u$ and $v$ be distinct non-consecutive vertices on $F$. As $G$ is edge-maximal, $u$ and $v$ are adjacent. A \emph{flip} of the edge $uv$ is a modification of the embedding by re-embedding the edge $uv$ into another face, such as the face $F$.

\begin{proposition}\label{uvw distinct}
	Let $G$ be a 2-connected graph with an embedding on a surface $\Sigma$. Let $u,v,w$ be three consecutive vertices on the boundary of a face $F$, then $u,v,w$ are distinct.
\end{proposition}

\begin{proof}
	The edges $uv$ and $vw$ on the boundary are distinct as otherwise $v$ would be a cut vertex. So as $G$ is simple all three vertices must be distinct.
\end{proof}

\begin{lemma}\label{4-face}
	If a graph $G$ is Euler impure on some surface $\Sigma$, then there exists an embedding of $G$ on $\Sigma$ with a 4-face.
\end{lemma}

\begin{proof}
	Consider an embedding of $G$. As $G$ does not triangulate $\Sigma$, we can find a face $F$ of size at least 4. If $F$ has size at least 5, and there are four consecutive distinct vertices $u,v,w,x$ on the boundary of $F$, then we can obtain an embedding of $G$ with a 4-face by flipping the edge $ux$ into $F$, so we may assume that this is not the case.
	
	Let $u_1,v_1,w_1,u_2,v_2,w_2,u_3,v_3$ be consecutive vertices appearing on the boundary of $F$ (allowing to traverse the same part of the boundary twice if $F$ has size less than 8). Notice that $G$ must be 2-connected as it is edge-maximal. Any three consecutive vertices must be distinct by Proposition \ref{uvw distinct}, and every fourth vertex must be a repeat of the first vertex by our previous observation, so we must have that $u=u_1=u_2=u_3$, $v=v_1=v_2=v_3$ and $w=w_1=w_2$, where $u,v,w$ are distinct vertices. Now as the edge $uv$ appears in the same direction along the boundary at both $u_1v_1$ and $u_3v_3$, we see that $F$ is in fact a 6-face. Furthermore the boundary of $F$ contains both sides of all three edges $uv,vw$ and $wu$. So $G$ embeds with just the single face $F$. But then $G$ is isomorphic to $K_3$, a contradiction.
\end{proof}

We will require a complete list of embeddings of $K_4$ with a 4-face for the projective-plane, torus and Klein bottle. See Figures \ref{P1}, \ref{T1}, \ref{K1} and \ref{K2} for these embeddings, the shaded region indicates a given 4-face. We draw the projective plane as a disk with opposite points on its boundary identified and we draw the torus and Klein bottle as squares with opposite sides suitably identified. Figures \ref{P1} and \ref{T1} illustrates the unique such embeddings in the projective plane and torus, respectively. Figures \ref{K1} and \ref{K2} illustrate the two embeddings on the Klein bottle \cite{jackson2000atlas}.

\begin{multicols}{2}
	\center{
	\scalebox{0.36}{       
		\begin{tikzpicture}[x=0.75pt,y=0.75pt,yscale=-1,xscale=1]
		
		\draw [line width=3.75]    (420.91,210) -- (330,119.09) ;

		\draw [line width=3.75]    (330,300.91) -- (420.91,210) ;

		\draw [line width=3.75]    (239.09,210) -- (330,119.09) ;

		\draw [line width=3.75]    (239.09,210) -- (330,300.91) ;

		\draw  [line width=0.75]  (130,210) .. controls (130,99.54) and (219.54,10) .. (330,10) .. controls (440.46,10) and (530,99.54) .. (530,210) .. controls (530,320.46) and (440.46,410) .. (330,410) .. controls (219.54,410) and (130,320.46) .. (130,210) -- cycle ;
		\draw [line width=3.75]    (330,10) -- (330.29,58.8) -- (330.29,119.41) ;

		\draw [line width=3.75]    (330,410) -- (330,361.52) -- (330,300.91) ;

		\draw [line width=3.75]    (130,210) -- (178.48,210) -- (239.09,210) ;

		\draw [line width=3.75]    (530,210) -- (481.52,210) -- (420.91,210) ;

		\draw  [draw opacity=0][fill={rgb, 255:red, 0; green, 0; blue, 0 }  ,fill opacity=0.25 ][line width=2.25]  (329.71,119.41) -- (420.35,210.02) -- (329.73,300.61) -- (239.09,210) -- cycle ;
		\draw  [draw opacity=0] (120,0) -- (540,0) -- (540,420) -- (120,420) -- cycle ;

		\draw (330,119.09) node[VertexStyle, font={\bfseries\Huge}] (0) {a} ;
		
		\draw (420.91,210) node[VertexStyle, font={\bfseries\Huge}] (0) {b} ;
		
		\draw (330,300.91) node[VertexStyle, font={\bfseries\Huge}] (0) {c} ;
		
		\draw (239.09,210) node[VertexStyle, font={\bfseries\Huge}] (0) {d} ;

		\end{tikzpicture}
		
	}
\captionof{figure}{}\label{P1}}

	\columnbreak
	
	\center{
	\scalebox{0.35}{

		\begin{tikzpicture}[x=0.75pt,y=0.75pt,yscale=-1,xscale=1]
		
		\draw   (130,10) -- (530,10) -- (530,410) -- (130,410) -- cycle ;
		\draw  [color={rgb, 255:red, 0; green, 0; blue, 0 }  ,draw opacity=1 ][fill={rgb, 255:red, 0; green, 0; blue, 0 }  ,fill opacity=1 ] (120,230) -- (130,210) -- (140,230) -- (130,220) -- cycle ;
		\draw  [color={rgb, 255:red, 0; green, 0; blue, 0 }  ,draw opacity=1 ][fill={rgb, 255:red, 0; green, 0; blue, 0 }  ,fill opacity=1 ] (520,230) -- (530,210) -- (540,230) -- (530,220) -- cycle ;
		\draw  [color={rgb, 255:red, 0; green, 0; blue, 0 }  ,draw opacity=1 ][fill={rgb, 255:red, 0; green, 0; blue, 0 }  ,fill opacity=1 ] (310,0) -- (330,10) -- (310,20) -- (320,10) -- cycle ;
		\draw  [color={rgb, 255:red, 0; green, 0; blue, 0 }  ,draw opacity=1 ][fill={rgb, 255:red, 0; green, 0; blue, 0 }  ,fill opacity=1 ] (290,0) -- (310,10) -- (290,20) -- (300,10) -- cycle ;
		\draw  [color={rgb, 255:red, 0; green, 0; blue, 0 }  ,draw opacity=1 ][fill={rgb, 255:red, 0; green, 0; blue, 0 }  ,fill opacity=1 ] (310,400) -- (330,410) -- (310,420) -- (320,410) -- cycle ;
		\draw  [color={rgb, 255:red, 0; green, 0; blue, 0 }  ,draw opacity=1 ][fill={rgb, 255:red, 0; green, 0; blue, 0 }  ,fill opacity=1 ] (290,400) -- (310,410) -- (290,420) -- (300,410) -- cycle ;
		\draw [line width=3.75]    (130,350) -- (190,410) ;

		\draw [line width=3.75]    (130,70) -- (130,350) ;

		\draw [line width=3.75]    (530,70) -- (530,350) ;

		\draw [line width=3.75]    (190,10) -- (280,10) -- (470,10) ;

		\draw [line width=3.75]    (190,410) -- (470,410) ;

		\draw  [draw opacity=0][fill={rgb, 255:red, 0; green, 0; blue, 0 }  ,fill opacity=0.25 ] (130,350) -- (190,410) -- (130,410) -- cycle ;
		\draw [line width=3.75]    (470,10) -- (530,70) ;

		\draw [line width=3.75]    (470,410) -- (530,350) ;

		\draw [line width=3.75]    (130,70) -- (190,10) ;

		\draw  [draw opacity=0][fill={rgb, 255:red, 0; green, 0; blue, 0 }  ,fill opacity=0.25 ] (530,70) -- (470,10) -- (530,10) -- cycle ;
		\draw  [draw opacity=0][fill={rgb, 255:red, 0; green, 0; blue, 0 }  ,fill opacity=0.25 ] (190,10) -- (130,70) -- (130,10) -- cycle ;
		\draw  [draw opacity=0][fill={rgb, 255:red, 0; green, 0; blue, 0 }  ,fill opacity=0.25 ] (470,410) -- (530,350) -- (530,410) -- cycle ;

		\draw (130,350) node[VertexStyle, font={\bfseries\Huge}] (0) {a} ;
		
		\draw (190,410) node[VertexStyle, font={\bfseries\Huge}] (0) {b} ;
		
		\draw (130,70) node[VertexStyle, font={\bfseries\Huge}] (0) {c} ;
		
		\draw (530,70) node[VertexStyle, font={\bfseries\Huge}] (0) {c} ;
		
		\draw (530,350) node[VertexStyle, font={\bfseries\Huge}] (0) {a} ;
		
		\draw (470,10) node[VertexStyle, font={\bfseries\Huge}] (0) {d} ;
		
		\draw (470,410) node[VertexStyle, font={\bfseries\Huge}] (0) {d} ;
		
		\draw (190,10) node[VertexStyle, font={\bfseries\Huge}] (0) {b} ;

		\end{tikzpicture}	
		
	}
\captionof{figure}{}\label{T1}}

\end{multicols}

\begin{multicols}{2}
	\center{
	\scalebox{0.35}{       
		\begin{tikzpicture}[x=0.75pt,y=0.75pt,yscale=-1,xscale=1]
		
		\draw   (130,10) -- (530,10) -- (530,410) -- (130,410) -- cycle ;
		\draw  [color={rgb, 255:red, 0; green, 0; blue, 0 }  ,draw opacity=1 ][fill={rgb, 255:red, 0; green, 0; blue, 0 }  ,fill opacity=1 ] (120,230) -- (130,210) -- (140,230) -- (130,220) -- cycle ;
		\draw  [color={rgb, 255:red, 0; green, 0; blue, 0 }  ,draw opacity=1 ][fill={rgb, 255:red, 0; green, 0; blue, 0 }  ,fill opacity=1 ] (540,190) -- (530,210) -- (520,190) -- (530,200) -- cycle ;
		\draw  [color={rgb, 255:red, 0; green, 0; blue, 0 }  ,draw opacity=1 ][fill={rgb, 255:red, 0; green, 0; blue, 0 }  ,fill opacity=1 ] (310,0) -- (330,10) -- (310,20) -- (320,10) -- cycle ;
		\draw  [color={rgb, 255:red, 0; green, 0; blue, 0 }  ,draw opacity=1 ][fill={rgb, 255:red, 0; green, 0; blue, 0 }  ,fill opacity=1 ] (290,0) -- (310,10) -- (290,20) -- (300,10) -- cycle ;
		\draw  [color={rgb, 255:red, 0; green, 0; blue, 0 }  ,draw opacity=1 ][fill={rgb, 255:red, 0; green, 0; blue, 0 }  ,fill opacity=1 ] (310,400) -- (330,410) -- (310,420) -- (320,410) -- cycle ;
		\draw  [color={rgb, 255:red, 0; green, 0; blue, 0 }  ,draw opacity=1 ][fill={rgb, 255:red, 0; green, 0; blue, 0 }  ,fill opacity=1 ] (290,400) -- (310,410) -- (290,420) -- (300,410) -- cycle ;
		\draw [line width=3.75]    (190,350) -- (470,350) ;

		\draw [line width=3.75]    (130,350) -- (190,350) ;

		\draw [line width=3.75]    (190,350) -- (190,410) ;

		\draw  [draw opacity=0][fill={rgb, 255:red, 0; green, 0; blue, 0 }  ,fill opacity=0.25 ] (130,350) -- (190,350) -- (190,410) -- (130,410) -- cycle ;
		\draw [line width=3.75]    (470,350) -- (470,410) ;

		\draw [line width=3.75]    (470,350) -- (530,350) ;

		\draw  [draw opacity=0][fill={rgb, 255:red, 0; green, 0; blue, 0 }  ,fill opacity=0.25 ] (470,350) -- (530,350) -- (530,410) -- (470,410) -- cycle ;
		\draw [line width=3.75]    (470,70) -- (190,70) ;

		\draw [line width=3.75]    (530,70) -- (470,70) ;

		\draw [line width=3.75]    (470,70) -- (470,10) ;

		\draw  [draw opacity=0][fill={rgb, 255:red, 0; green, 0; blue, 0 }  ,fill opacity=0.25 ] (530,70) -- (470,70) -- (470,10) -- (530,10) -- cycle ;
		\draw [line width=3.75]    (190,70) -- (190,10) ;

		\draw [line width=3.75]    (190,70) -- (130,70) ;

		\draw  [draw opacity=0][fill={rgb, 255:red, 0; green, 0; blue, 0 }  ,fill opacity=0.25 ] (190,70) -- (130,70) -- (130,10) -- (190,10) -- cycle ;
		
		\draw (190,350) node[VertexStyle, font={\bfseries\Huge}] (0) {c} ;
		
		\draw (470,350) node[VertexStyle, font={\bfseries\Huge}] (0) {a} ;
		
		\draw (190,70) node[VertexStyle, font={\bfseries\Huge}] (0) {b} ;
		
		\draw (470,70) node[VertexStyle, font={\bfseries\Huge}] (0) {d} ;
		
		\tikzset{VertexStyle/.style = {shape= circle, inner sep= 0pt, outer sep= 0pt, minimum size = 30 pt, }}  
		
		\draw (190,10) node[VertexStyle, font={\bfseries\Huge}] (0) {} ;
		
	\end{tikzpicture}
		
	}
	\captionof{figure}{}\label{K1}}

	\columnbreak

	\center{
	\scalebox{0.35}{
		\begin{tikzpicture}[x=0.75pt,y=0.75pt,yscale=-1,xscale=1]
		
		\draw   (130,10) -- (530,10) -- (530,410) -- (130,410) -- cycle ;
		\draw  [color={rgb, 255:red, 0; green, 0; blue, 0 }  ,draw opacity=1 ][fill={rgb, 255:red, 0; green, 0; blue, 0 }  ,fill opacity=1 ] (120,230) -- (130,210) -- (140,230) -- (130,220) -- cycle ;
		\draw  [color={rgb, 255:red, 0; green, 0; blue, 0 }  ,draw opacity=1 ][fill={rgb, 255:red, 0; green, 0; blue, 0 }  ,fill opacity=1 ] (540,190) -- (530,210) -- (520,190) -- (530,200) -- cycle ;
		\draw  [color={rgb, 255:red, 0; green, 0; blue, 0 }  ,draw opacity=1 ][fill={rgb, 255:red, 0; green, 0; blue, 0 }  ,fill opacity=1 ] (310,0) -- (330,10) -- (310,20) -- (320,10) -- cycle ;
		\draw  [color={rgb, 255:red, 0; green, 0; blue, 0 }  ,draw opacity=1 ][fill={rgb, 255:red, 0; green, 0; blue, 0 }  ,fill opacity=1 ] (290,0) -- (310,10) -- (290,20) -- (300,10) -- cycle ;
		\draw  [color={rgb, 255:red, 0; green, 0; blue, 0 }  ,draw opacity=1 ][fill={rgb, 255:red, 0; green, 0; blue, 0 }  ,fill opacity=1 ] (310,400) -- (330,410) -- (310,420) -- (320,410) -- cycle ;
		\draw  [color={rgb, 255:red, 0; green, 0; blue, 0 }  ,draw opacity=1 ][fill={rgb, 255:red, 0; green, 0; blue, 0 }  ,fill opacity=1 ] (290,400) -- (310,410) -- (290,420) -- (300,410) -- cycle ;
		\draw [line width=3.75]    (130,350) -- (190,410) ;

		\draw [line width=3.75]    (130,70) -- (130,350) ;

		\draw [line width=3.75]    (530,70) -- (530,350) ;

		\draw [line width=3.75]    (190,10) -- (280,10) -- (470,10) ;

		\draw [line width=3.75]    (190,410) -- (470,410) ;

		\draw  [draw opacity=0][fill={rgb, 255:red, 0; green, 0; blue, 0 }  ,fill opacity=0.25 ] (130,350) -- (190,410) -- (130,410) -- cycle ;
		\draw [line width=3.75]    (470,10) -- (530,70) ;

		\draw [line width=3.75]    (470,410) -- (530,350) ;

		\draw [line width=3.75]    (130,70) -- (190,10) ;

		\draw  [draw opacity=0][fill={rgb, 255:red, 0; green, 0; blue, 0 }  ,fill opacity=0.25 ] (530,70) -- (470,10) -- (530,10) -- cycle ;
		\draw  [draw opacity=0][fill={rgb, 255:red, 0; green, 0; blue, 0 }  ,fill opacity=0.25 ] (190,10) -- (130,70) -- (130,10) -- cycle ;
		\draw  [draw opacity=0][fill={rgb, 255:red, 0; green, 0; blue, 0 }  ,fill opacity=0.25 ] (470,410) -- (530,350) -- (530,410) -- cycle ;
		
		\draw (130,350) node[VertexStyle, font={\bfseries\Huge}] (0) {a} ;
		
		\draw (190,410) node[VertexStyle, font={\bfseries\Huge}] (0) {b} ;
		
		\draw (130,70) node[VertexStyle, font={\bfseries\Huge}] (0) {c} ;
		
		\draw (530,70) node[VertexStyle, font={\bfseries\Huge}] (0) {a} ;
		
		\draw (530,350) node[VertexStyle, font={\bfseries\Huge}] (0) {c} ;
		
		\draw (470,10) node[VertexStyle, font={\bfseries\Huge}] (0) {d} ;
		
		\draw (470,410) node[VertexStyle, font={\bfseries\Huge}] (0) {d} ;
		
		\draw (190,10) node[VertexStyle, font={\bfseries\Huge}] (0) {b} ;

		\end{tikzpicture}

	}
	\captionof{figure}{}\label{K2}}
\end{multicols}

For completeness, we remark that there is no such embedding in the plane, so by Lemma \ref{4-face} there is no Euler impure planar graph.

We now briefly outline our approach to Theorems \ref{projectivePlane}, \ref{klein} and \ref{torus}. We will start by supposing that an Euler impure graph $G$ exists and consider an embedding with a 4-face inducing $K_4$. Then we begin to deduce how the vertices, edges and faces must embed until we either arrive at a contradiction or deduce that $G$ is isomorphic to some known Euler impure graph. Shaded regions in the figures indicate faces while unshaded regions could still contain any number of additional edges or vertices.

We begin with the projective plane.

\begin{proof}[Proof of Theorem \ref{projectivePlane}]
	Suppose that $G$ is an Euler impure projective-planar graph. Then by Lemma \ref{4-face} there exists an embedding of $G$ with a 4-face $abcd$ whose vertices induce $K_4$. Such a 4-face has a unique embedding (see Figure \ref{P1}).
	
	At least one of $bcad$ or $abdc$ must not be a face, else $G\cong K_4$. Also at least one must be a face, as otherwise we could add an edge from a vertex inside $bcad$ to a vertex inside $abdc$ after flipping the edge $bd$ into $abcd$. So without loss of generality we may assume that $bcad$ is a face and $abdc$ contains all remaining vertices of $G$.
	
	Let $e$ be the vertex adjacent to $b$ appearing first anticlockwise from the edge $bd$. Then as $bd$ is flippable, $e$ must be adjacent to $a$, $c$ and $d$ as well (see Figure \ref{P2}).

\begin{center}

	\center{
	\scalebox{0.36}{
	
	\begin{tikzpicture}[x=0.75pt,y=0.75pt,yscale=-1,xscale=1]
	
	\draw  [draw opacity=0][fill={rgb, 255:red, 0; green, 0; blue, 0 }  ,fill opacity=0.25 ] (130,210.32) .. controls (130,99.86) and (219.54,10.32) .. (330,10.32) -- (330,210.32) -- cycle ;
	\draw [line width=3.75]    (420.91,210) -- (330,119.09) ;

	\draw [line width=3.75]    (330,300.91) -- (420.91,210) ;

	\draw [line width=3.75]    (239.09,210) -- (330,119.09) ;

	\draw [line width=3.75]    (239.09,210) -- (330,300.91) ;

	\draw  [line width=0.75]  (130,210) .. controls (130,99.54) and (219.54,10) .. (330,10) .. controls (440.46,10) and (530,99.54) .. (530,210) .. controls (530,320.46) and (440.46,410) .. (330,410) .. controls (219.54,410) and (130,320.46) .. (130,210) -- cycle ;
	\draw [line width=3.75]    (330,10) -- (330.29,58.8) -- (330.29,119.41) ;

	\draw [line width=3.75]    (330,410) -- (330,361.52) -- (330,300.91) ;

	\draw [line width=3.75]    (130,210) -- (178.48,210) -- (239.09,210) ;

	\draw [line width=3.75]    (530,210) -- (481.52,210) -- (420.91,210) ;

	\draw [line width=3.75]    (420.91,210) -- (470,70) ;

	\draw [line width=3.75]    (330.29,119.41) -- (470,70) ;

	\draw [line width=3.75]    (190.02,350.02) -- (329.73,300.61) ;

	\draw [line width=3.75]    (190.02,350.02) -- (239.11,210.02) ;

	\draw  [draw opacity=0][fill={rgb, 255:red, 0; green, 0; blue, 0 }  ,fill opacity=0.25 ] (530,210) .. controls (530,210) and (530,210) .. (530,210) .. controls (530,320.46) and (440.46,410) .. (330,410) -- (330,210) -- cycle ;
	\draw  [draw opacity=0][fill={rgb, 255:red, 0; green, 0; blue, 0 }  ,fill opacity=0.25 ] (330,119.41) -- (419.71,210) -- (330,210) -- cycle ;
	\draw  [draw opacity=0][fill={rgb, 255:red, 0; green, 0; blue, 0 }  ,fill opacity=0.25 ] (330,300.91) -- (240.29,210.32) -- (330,210.32) -- cycle ;
	\draw  [draw opacity=0] (120,0) -- (540,0) -- (540,420) -- (120,420) -- cycle ;
	
	\draw (330,119.09) node[VertexStyle, font={\bfseries\Huge}] (0) {a} ;
	
	\draw (420.91,210) node[VertexStyle, font={\bfseries\Huge}] (0) {b} ;
	
	\draw (330,300.91) node[VertexStyle, font={\bfseries\Huge}] (0) {c} ;
	
	\draw (239.09,210) node[VertexStyle, font={\bfseries\Huge}] (0) {d} ;
	
	\draw (190.02,350.02) node[VertexStyle, font={\bfseries\Huge}] (0) {e} ;
	
	\draw (470,70) node[VertexStyle, font={\bfseries\Huge}] (0) {e} ;

	\end{tikzpicture}
	
}

\captionof{figure}{}\label{P2}}

\end{center}

Now as each edge $ab$, $bd$, $dc$ and $da$ is flippable, similarly to before we see that each of the triangles $aeb$, $bed$, $dec$ and $dea$ must be faces. Therefore $G\cong K_5$, a contradiction.
\end{proof}

For the torus and Klein bottle we will continue to repeatedly use this idea of looking for edges we might add after flipping edges.

We prove a technical Lemma which we shall apply when characterizing the Euler impure graphs on torus and then again the Klein bottle. A closed surface with boundary is a \emph{$2$-cell} if it is homeomorphic to the disk $\mathbb{D}$.

\begin{lemma}\label{>--<}
	Suppose that $G$ is an Euler impure graph on some surface $\Sigma$ with an embedding having a closed walk $uvwxy_1y_2z_1z_2$ of length 8 bounding a 
	2-cell $C$ (to the right of this walk). Further suppose that $C$ contains no additional edges between vertices of the walk, $C$ contains at least two additional vertices, $\{y_1,y_2\}=\{u,v\}$ and $\{z_1,z_2\}=\{w,x\}$, $G[\{u,v,w,x\}]=K_4$, and the edge $uv$ is flippable to some face outside $C$. Then there are adjacent vertices $e$ and $f$ inside 
	$C$ such that $uve$ and $y_1y_2f$ are faces (see Figure \ref{L5}) after possibly flipping some subset of the edges $\{eu,ev,fu,fv\}$.
\end{lemma}

\begin{proof}
	If $u$ is connected to at least one vertex by an edge inside $C$, then let $e$ be the first such neighbour appearing clockwise after $v$. 
If $u$ has no such neighbor in $C$, 	
there must exist some vertex $e$ in $C$ lying on the same face as $z_2uv$. As $G$ is edge-maximal, $u$ is adjacent to $e$. We now flip the edge $ue$ so that it connects $u$ and $e$ in $C$. Now $v$ is adjacent to $e$ as $G$ is edge-maximal. By possibly flipping the edge $ve$ we get $uve$ as a face inside $C$ (see Figure \ref{L2}).
	
	Suppose for sake of contradiction that there is no edge incident to $y_2$ inside $C$, and further that the face $F$ in $C$ which $y_1y_2z_1$ lies on contains only vertices from $\{e,u,v,w,x,y_1,y_2,z_1,z_2\}$. Now $e$ must be a vertex of this face and as $e$ is adjacent to both $u$ and $v$, $x$ is also on this face. 
	By a possible edge flip, we may assume that there is an edge between $e$ and $z_1$. Now there are two cases to consider, the first being that there is an edge between $e$ and $z_2$, in which case $xy_1y_2z_1evw$ is a face (see Figure \ref{L3}, remember that $\{z_1,z_2\}=\{w,x\}$). In this case both $z_1z_2e$ and $ez_2u$ must be faces as edges $ez_1$ and $ez_2$ can be flipped to $ew$ and $ex$, contradicting the fact that there should be at least two additional vertices. The second case is that there is an edge between $e$ and either $w$ or $x$, in this case $y_1y_2z_1ewx$ or $y_1y_2z_1ex$ respectively are faces (see Figure \ref{L4} for when there is an edge between $e$ and $w$). Now, as edges $eu$ and $ev$ can be flipped to $ey_1$ and $ey_2$, we see that $z_1z_2ue$ is a face and either $evw$ or $evwx$ respectively are faces. We again have a contradiction as there should be at least two additional vertices.
	
	Now suppose again that there is no edge incident to $y_2$ inside $C$. By the previous discussion, the face in $C$ which $y_1y_2z_1$ lies on must contain some vertex $f$ distinct from the vertices $\{e,u,v,w,x,y_1,y_2,z_1,z_2\}$. As $G$ is edge-maximal, $y_2$ and $f$ must be adjacent, and so we can flip $y_2f$ into that face.
	
	Hence we may assume that there exists some edge incident to $y_2$ inside $C$. Let $y_2f$ be the first such edge appearing anticlockwise after $y_1$. Now similarly to before, as $G$ is edge-maximal, $f$ must be adjacent to $y_1$. By possibly flipping the edge $y_1f$ we get $y_1y_2f$ as a face. Now finally, as the edge $uv$ is flippable, $e$ must be adjacent to $f$ as required (and as depicted in Figure \ref{L5}).
\end{proof}

\begin{multicols}{2}
	\center{
	\scalebox{0.35}{       
		
		\begin{tikzpicture}[x=0.75pt,y=0.75pt,yscale=-1,xscale=1]
		
		\draw [line width=3.75]    (130,350) -- (190,410) ;

		\draw [line width=3.75]    (130,70) -- (130,350) ;

		\draw [line width=3.75]    (530,70) -- (530,350) ;

		\draw [line width=3.75]    (470,10) -- (530,70) ;

		\draw [line width=3.75]    (470,410) -- (530,350) ;

		\draw [line width=3.75]    (130,70) -- (169.75,30.25) -- (190,10) ;

		\draw [line width=3.75]    (190,10) -- (280,10) -- (470,10) ;

		\draw [line width=3.75]    (190,410) -- (280,410) -- (470,410) ;

		\draw [line width=3.75]    (190,210) -- (130,70) ;

		\draw [line width=3.75]    (190,210) -- (130,350) ;

		\draw  [draw opacity=0][fill={rgb, 255:red, 0; green, 0; blue, 0 }  ,fill opacity=0.25 ] (190,210) -- (130,350) -- (130,70) -- cycle ;
		
		\draw (130,350) node[VertexStyle, font={\bfseries\Huge}] (0) {u} ;
		
		\draw (190,410) node[VertexStyle, font={\bfseries\Huge}] (0) {z$_2$} ;
		
		\draw (130,70) node[VertexStyle, font={\bfseries\Huge}] (0) {v} ;
		
		\draw (530,70) node[VertexStyle, font={\bfseries\Huge}] (0) {y$_1$} ;
		
		\draw (530,350) node[VertexStyle, font={\bfseries\Huge}] (0) {y$_2$} ;
		
		\draw (470,10) node[VertexStyle, font={\bfseries\Huge}] (0) {x} ;
		
		\draw (470,410) node[VertexStyle, font={\bfseries\Huge}] (0) {z$_1$} ;
		
		\draw (190,10) node[VertexStyle, font={\bfseries\Huge}] (0) {w} ;
		
		\draw (190,210) node[VertexStyle, font={\bfseries\Huge}] (0) {e} ;

		\end{tikzpicture}

	}
	
	\captionof{figure}{}\label{L2}}

	\columnbreak
	
	\center{
	\scalebox{0.35}{
		
		\begin{tikzpicture}[x=0.75pt,y=0.75pt,yscale=-1,xscale=1]
		
		\draw [line width=3.75]    (130,350) -- (190,410) ;

		\draw [line width=3.75]    (130,70) -- (130,350) ;

		\draw [line width=3.75]    (530,70) -- (530,350) ;

		\draw [line width=3.75]    (470,10) -- (530,70) ;

		\draw [line width=3.75]    (470,410) -- (530,350) ;

		\draw [line width=3.75]    (130,70) -- (169.75,30.25) -- (190,10) ;

		\draw [line width=3.75]    (190,10) -- (280,10) -- (470,10) ;

		\draw [line width=3.75]    (190,410) -- (280,410) -- (470,410) ;

		\draw [line width=3.75]    (190,210) -- (130,70) ;

		\draw [line width=3.75]    (190,210) -- (130,350) ;

		\draw  [draw opacity=0][fill={rgb, 255:red, 0; green, 0; blue, 0 }  ,fill opacity=0.25 ] (190,210) -- (130,350) -- (130,70) -- cycle ;
		\draw [line width=3.75]    (190,210) -- (190,410) ;

		\draw [line width=3.75]    (190,210) -- (470,410) ;

		\draw  [draw opacity=0][fill={rgb, 255:red, 0; green, 0; blue, 0 }  ,fill opacity=0.25 ] (130,71) -- (190,10) -- (190,210) -- cycle ;
		\draw  [draw opacity=0][fill={rgb, 255:red, 0; green, 0; blue, 0 }  ,fill opacity=0.25 ] (470.5,410) -- (390,350) -- (530,350) -- cycle ;
		\draw  [draw opacity=0][fill={rgb, 255:red, 0; green, 0; blue, 0 }  ,fill opacity=0.25 ] (390,350) -- (190,210) -- (390,210) -- cycle ;
		\draw  [draw opacity=0][fill={rgb, 255:red, 0; green, 0; blue, 0 }  ,fill opacity=0.25 ] (470,10) -- (530,70) -- (470,70) -- cycle ;
		\draw  [draw opacity=0][fill={rgb, 255:red, 0; green, 0; blue, 0 }  ,fill opacity=0.25 ] (190,10) -- (470,10) -- (470,210) -- (190,210) -- cycle ;
		\draw  [draw opacity=0][fill={rgb, 255:red, 0; green, 0; blue, 0 }  ,fill opacity=0.25 ] (390,210) -- (530,210) -- (530,350) -- (390,350) -- cycle ;
		\draw  [draw opacity=0][fill={rgb, 255:red, 0; green, 0; blue, 0 }  ,fill opacity=0.25 ] (470,70) -- (530,70) -- (530,210) -- (470,210) -- cycle ;
		
		\draw (130,350) node[VertexStyle, font={\bfseries\Huge}] (0) {u} ;
		
		\draw (190,410) node[VertexStyle, font={\bfseries\Huge}] (0) {z$_2$} ;
		
		\draw (130,70) node[VertexStyle, font={\bfseries\Huge}] (0) {v} ;
		
		\draw (530,70) node[VertexStyle, font={\bfseries\Huge}] (0) {y$_1$} ;
		
		\draw (530,350) node[VertexStyle, font={\bfseries\Huge}] (0) {y$_2$} ;
		
		\draw (470,10) node[VertexStyle, font={\bfseries\Huge}] (0) {x} ;
		
		\draw (470,410) node[VertexStyle, font={\bfseries\Huge}] (0) {z$_1$} ;
		
		\draw (190,10) node[VertexStyle, font={\bfseries\Huge}] (0) {w} ;
		
		\draw (190,210) node[VertexStyle, font={\bfseries\Huge}] (0) {e} ;

		\end{tikzpicture}

	}
	
	\captionof{figure}{}\label{L3}}

\end{multicols}

\begin{multicols}{2}
	\center{
	\scalebox{0.35}{       
		\begin{tikzpicture}[x=0.75pt,y=0.75pt,yscale=-1,xscale=1]
		
		\draw [line width=3.75]    (130,350) -- (190,410) ;

		\draw [line width=3.75]    (130,70) -- (130,350) ;

		\draw [line width=3.75]    (530,70) -- (530,350) ;

		\draw [line width=3.75]    (470,10) -- (530,70) ;

		\draw [line width=3.75]    (470,410) -- (530,350) ;

		\draw [line width=3.75]    (130,70) -- (169.75,30.25) -- (190,10) ;

		\draw [line width=3.75]    (190,10) -- (280,10) -- (470,10) ;

		\draw [line width=3.75]    (190,410) -- (280,410) -- (470,410) ;

		\draw [line width=3.75]    (190,210) -- (130,70) ;

		\draw [line width=3.75]    (190,210) -- (130,350) ;

		\draw  [draw opacity=0][fill={rgb, 255:red, 0; green, 0; blue, 0 }  ,fill opacity=0.25 ] (190,210) -- (130,350) -- (130,70) -- cycle ;
		\draw [line width=3.75]    (190,210) -- (190,10) ;

		\draw [line width=3.75]    (190,210) -- (470,410) ;

		\draw  [draw opacity=0][fill={rgb, 255:red, 0; green, 0; blue, 0 }  ,fill opacity=0.25 ] (470.5,410) -- (390,350) -- (530,350) -- cycle ;
		\draw  [draw opacity=0][fill={rgb, 255:red, 0; green, 0; blue, 0 }  ,fill opacity=0.25 ] (390,350) -- (190,210) -- (390,210) -- cycle ;
		\draw  [draw opacity=0][fill={rgb, 255:red, 0; green, 0; blue, 0 }  ,fill opacity=0.25 ] (470,10) -- (530,70) -- (470,70) -- cycle ;
		\draw  [draw opacity=0][fill={rgb, 255:red, 0; green, 0; blue, 0 }  ,fill opacity=0.25 ] (190,10) -- (470,10) -- (470,210) -- (190,210) -- cycle ;
		\draw  [draw opacity=0][fill={rgb, 255:red, 0; green, 0; blue, 0 }  ,fill opacity=0.25 ] (390,210) -- (530,210) -- (530,350) -- (390,350) -- cycle ;
		\draw  [draw opacity=0][fill={rgb, 255:red, 0; green, 0; blue, 0 }  ,fill opacity=0.25 ] (470,70) -- (530,70) -- (530,210) -- (470,210) -- cycle ;
		
		\draw (130,350) node[VertexStyle, font={\bfseries\Huge}] (0) {u} ;
		
		\draw (190,410) node[VertexStyle, font={\bfseries\Huge}] (0) {z$_2$} ;
		
		\draw (130,70) node[VertexStyle, font={\bfseries\Huge}] (0) {v} ;
		
		\draw (530,70) node[VertexStyle, font={\bfseries\Huge}] (0) {y$_1$} ;
		
		\draw (530,350) node[VertexStyle, font={\bfseries\Huge}] (0) {y$_2$} ;
		
		\draw (470,10) node[VertexStyle, font={\bfseries\Huge}] (0) {x} ;
		
		\draw (470,410) node[VertexStyle, font={\bfseries\Huge}] (0) {z$_1$} ;
		
		\draw (190,10) node[VertexStyle, font={\bfseries\Huge}] (0) {w} ;
		
		\draw (190,210) node[VertexStyle, font={\bfseries\Huge}] (0) {e} ;

		\end{tikzpicture}
		
	}
	
	\captionof{figure}{}\label{L4}}

	\columnbreak

	\center{
	\scalebox{0.35}{
		
	\begin{tikzpicture}[x=0.75pt,y=0.75pt,yscale=-1,xscale=1]
	
	\draw [line width=3.75]    (130,350) -- (190,410) ;

	\draw [line width=3.75]    (130,70) -- (130,350) ;

	\draw [line width=3.75]    (530,70) -- (530,350) ;

	\draw [line width=3.75]    (470,10) -- (530,70) ;

	\draw [line width=3.75]    (470,410) -- (530,350) ;

	\draw [line width=3.75]    (130,70) -- (169.75,30.25) -- (190,10) ;

	\draw [line width=3.75]    (190,10) -- (280,10) -- (470,10) ;

	\draw [line width=3.75]    (190,410) -- (280,410) -- (470,410) ;

	\draw [line width=3.75]    (190,210) -- (130,70) ;

	\draw [line width=3.75]    (190,210) -- (130,350) ;

	\draw  [draw opacity=0][fill={rgb, 255:red, 0; green, 0; blue, 0 }  ,fill opacity=0.25 ] (190,210) -- (130,350) -- (130,70) -- cycle ;
	\draw [line width=3.75]    (470,210) -- (530,350) ;

	\draw [line width=3.75]    (470,210) -- (530,70) ;

	\draw  [draw opacity=0][fill={rgb, 255:red, 0; green, 0; blue, 0 }  ,fill opacity=0.25 ] (470,210) -- (530,70) -- (530,350) -- cycle ;
	
	\draw [line width=3.75]    (190,210) -- (280,210) -- (470,210) ;

	\draw (130,350) node[VertexStyle, font={\bfseries\Huge}] (0) {u} ;
	
	\draw (190,410) node[VertexStyle, font={\bfseries\Huge}] (0) {z$_2$} ;
	
	\draw (130,70) node[VertexStyle, font={\bfseries\Huge}] (0) {v} ;
	
	\draw (530,70) node[VertexStyle, font={\bfseries\Huge}] (0) {y$_1$} ;
	
	\draw (530,350) node[VertexStyle, font={\bfseries\Huge}] (0) {y$_2$} ;
	
	\draw (470,10) node[VertexStyle, font={\bfseries\Huge}] (0) {x} ;
	
	\draw (470,410) node[VertexStyle, font={\bfseries\Huge}] (0) {z$_1$} ;
	
	\draw (190,10) node[VertexStyle, font={\bfseries\Huge}] (0) {w} ;
	
	\draw (190,210) node[VertexStyle, font={\bfseries\Huge}] (0) {e} ;
	
	\draw (470,210) node[VertexStyle, font={\bfseries\Huge}] (0) {f} ;

	\end{tikzpicture}

	}
	
	\captionof{figure}{}\label{L5}}

\end{multicols}

\begin{lemma}\label{T+K lemma}
	Suppose that $G$ is an Euler impure graph on some surface $\Sigma$ with an embedding having a closed walk $acbdy_1y_2db$ of length 8 bounding a 2-cell $C$ 
	containing no additional edges between vertices of the walk, at least three additional vertices and such that $\{y_1,y_2\}=\{a,c\}$ where $G[\{a,b,c,d\}]=K_4$ with the edges $ac$ and $bd$ being flippable to some face outside $C$. Then there exists an embedding with two pairs of adjacent vertices $e,f$ and $x,y$ with $|\{e,f\}\cap \{x,y\}|=1$ in the interior of $C$ such that $ace$, $y_1y_2f$, $bdx$ and $dby$ are faces inside $C$.
\end{lemma}

\begin{proof}
	First we apply Lemma \ref{>--<} twice to obtain two pairs of adjacent vertices $e,f$ and $x,y$ such that $ace$, $y_1y_2f$, $bdx$, $dby$ are faces inside $C$. Clearly $\{e,f\}$ and $\{x,y\}$ can not be disjoint, so $|\{e,f\}\cap \{x,y\}|\ge1$.
	
	Suppose for sake of contradiction that $\{e,f\}=\{x,y\}$, then without loss of generality we may assume that $e=x$ and $f=y$ (as in Figure \ref{L6}). By flipping edges $ac$ then $ef$, we see that both $efba$ and $fedy_1$ must be faces. Now edges $be$ and $df$ can be flipped, and so $bec$ and $dfy_2$ are faces. But now this contradicts the number of vertices in the interior of $C$.
	
	Hence $|\{e,f\}\cap \{x,y\}|=1$ as required.
\end{proof}

\begin{center}
	\scalebox{0.35}{
		\begin{tikzpicture}[x=0.75pt,y=0.75pt,yscale=-1,xscale=1]
		
		\draw [line width=3.75]    (130,350) -- (190,410) ;

		\draw [line width=3.75]    (130,70) -- (130,350) ;

		\draw [line width=3.75]    (530,70) -- (530,350) ;

		\draw [line width=3.75]    (470,10) -- (530,70) ;

		\draw [line width=3.75]    (470,410) -- (530,350) ;

		\draw [line width=3.75]    (130,70) -- (169.75,30.25) -- (190,10) ;

		\draw [line width=3.75]    (190,10) -- (280,10) -- (470,10) ;

		\draw [line width=3.75]    (190,410) -- (280,410) -- (470,410) ;

		\draw [line width=3.75]    (190,210) -- (130,70) ;

		\draw [line width=3.75]    (190,210) -- (130,350) ;

		\draw  [draw opacity=0][fill={rgb, 255:red, 0; green, 0; blue, 0 }  ,fill opacity=0.25 ] (190,210) -- (130,350) -- (130,70) -- cycle ;
		\draw [line width=3.75]    (470,210) -- (530,350) ;

		\draw [line width=3.75]    (470,210) -- (530,70) ;

		\draw  [draw opacity=0][fill={rgb, 255:red, 0; green, 0; blue, 0 }  ,fill opacity=0.25 ] (470,210) -- (530,70) -- (530,350) -- cycle ;
		
		\draw [line width=3.75]    (190,210) -- (280,210) -- (470,210) ;

		\draw [line width=3.75]    (190,210) -- (190,10) ;

		\draw [line width=3.75]    (190,210) -- (470,10) ;

		\draw  [draw opacity=0][fill={rgb, 255:red, 0; green, 0; blue, 0 }  ,fill opacity=0.25 ] (470,10) -- (190,210) -- (190,10) -- cycle ;
		\draw [line width=3.75]    (470,210) -- (470,410) ;

		\draw [line width=3.75]    (470,210) -- (190,410) ;

		\draw  [draw opacity=0][fill={rgb, 255:red, 0; green, 0; blue, 0 }  ,fill opacity=0.25 ] (190,410) -- (470,210) -- (470,410) -- cycle ;

		\draw (130,350) node[VertexStyle, font={\bfseries\Huge}] (0) {a} ;
		
		\draw (190,410) node[VertexStyle, font={\bfseries\Huge}] (0) {b} ;
		
		\draw (130,70) node[VertexStyle, font={\bfseries\Huge}] (0) {c} ;
		
		\draw (530,70) node[VertexStyle, font={\bfseries\Huge}] (0) {y$_1$} ;
		
		\draw (530,350) node[VertexStyle, font={\bfseries\Huge}] (0) {y$_2$} ;
		
		\draw (470,10) node[VertexStyle, font={\bfseries\Huge}] (0) {d} ;
		
		\draw (470,410) node[VertexStyle, font={\bfseries\Huge}] (0) {d} ;
		
		\draw (190,10) node[VertexStyle, font={\bfseries\Huge}] (0) {b} ;
		
		\draw (190,210) node[VertexStyle, font={\bfseries\Huge}] (0) {e} ;
		
		\draw (470,210) node[VertexStyle, font={\bfseries\Huge}] (0) {f} ;

		\end{tikzpicture}
	}
\captionof{figure}{}\label{L6}

\end{center}

We are now prepared to characterize the Euler impure graphs on the torus and the Klein bottle. First the torus.

\begin{proof}[Proof of Theorem \ref{torus}]
	Suppose that $G$ is an Euler impure toroidal graph. Then by Lemma \ref{4-face} there exists an embedding of $G$ with a 4-face $abcd$ whose vertices induce $K_4$. There is a unique such embedding of $K_4$ (see Figure \ref{T1}).
	
	As $K_7$ is embeddable, the 2-cell $C$ 
	bounded by $acbdcadb$ must contain at least 4 additional vertices in its interior. Both edges $ac$ and $bd$ can be flipped into the 4-face, so we may apply Lemma \ref{T+K lemma} and without loss of generality deduce that there is an embedding in which $C$ contains vertices $e,f,g$ such that $e$ is adjacent to both $f$ and $g$ and $ace$, $caf$, $bdx$ and $dby$ are all faces (see Figure \ref{T4}).
	
	By possibly flipping the edge $df$ (if it is an edge), we may assume that $f$ is adjacent to some vertex $x$ appearing first clockwise after $e$ but before $c$. But as we can flip the edge $ac$ and then $ef$, we see that we must have $x=d$ and further that $fed$ is a face.
	
	As edges $bd$ and then $eg$ are flippable, we see that $aeg$ lies on a face. Furthermore as $b$ and $f$ are non-adjacent we must also have that  $a$ is adjacent to $g$, with $aeg$ being a face (see Figure \ref{T5}).
	
	Suppose that $gef$ lie on a face, then we can perform a sequence of edge flips $ac$, $ef$, $dg$ and add an edge from $b$ to either $f$, or some vertex lying in $fadge$, a contradiction. Hence $e$ is adjacent to some new vertex $h$ appearing first clockwise after $f$ and before $g$. Now $h$ must be adjacent to $f$ with $efh$ being a face. We can also flip edges $ac$ and then $ef$, so we see that $h$ must also adjacent to $d$. By flipping edges $bd$ and then $eg$, we also see that $h$ must be adjacent to $g$, with $ehg$ being a face. By the same two edge flips, $h$ must also be adjacent to $a$ (see Figure \ref{T6}).
	
	Now both $hdg$ and $dha$ must be faces as we can perform the sequence of edge flips $ac$, $ef$, $dh$. Then further $agb$ must be a face by the sequence of edge flips $ac$, $ef$, $dh$, $ag$. The disc bounded by $cbe$ must be a face by the sequence of edge flips $ac$, $ef$, $dh$, $ag$, $be$. Similarly $ahf$ must be a face by the sequence of edge flips $bd$, $eg$, $ah$ and then lastly, $dcf$ must also be a face by the sequence of edge flips $bd$, $eg$, $ah$, $df$.
	
	Hence, $G \cong K_8-E(C_5)$ (with the missing cycle $bhcgf$) which we know to be Euler impure \cite{harary1974maximal}.
\end{proof}

\begin{multicols}{3}
	\center{
	\scalebox{0.35}{
		
		\begin{tikzpicture}[x=0.75pt,y=0.75pt,yscale=-1,xscale=1]
		
%
		\draw   (130,10) -- (530,10) -- (530,410) -- (130,410) -- cycle ;
		\draw  [color={rgb, 255:red, 0; green, 0; blue, 0 }  ,draw opacity=1 ][fill={rgb, 255:red, 0; green, 0; blue, 0 }  ,fill opacity=1 ] (120,230) -- (130,210) -- (140,230) -- (130,220) -- cycle ;
		\draw  [color={rgb, 255:red, 0; green, 0; blue, 0 }  ,draw opacity=1 ][fill={rgb, 255:red, 0; green, 0; blue, 0 }  ,fill opacity=1 ] (520,230) -- (530,210) -- (540,230) -- (530,220) -- cycle ;
		\draw  [color={rgb, 255:red, 0; green, 0; blue, 0 }  ,draw opacity=1 ][fill={rgb, 255:red, 0; green, 0; blue, 0 }  ,fill opacity=1 ] (310,0) -- (330,10) -- (310,20) -- (320,10) -- cycle ;
		\draw  [color={rgb, 255:red, 0; green, 0; blue, 0 }  ,draw opacity=1 ][fill={rgb, 255:red, 0; green, 0; blue, 0 }  ,fill opacity=1 ] (290,0) -- (310,10) -- (290,20) -- (300,10) -- cycle ;
		\draw  [color={rgb, 255:red, 0; green, 0; blue, 0 }  ,draw opacity=1 ][fill={rgb, 255:red, 0; green, 0; blue, 0 }  ,fill opacity=1 ] (310,400) -- (330,410) -- (310,420) -- (320,410) -- cycle ;
		\draw  [color={rgb, 255:red, 0; green, 0; blue, 0 }  ,draw opacity=1 ][fill={rgb, 255:red, 0; green, 0; blue, 0 }  ,fill opacity=1 ] (290,400) -- (310,410) -- (290,420) -- (300,410) -- cycle ;

		\draw [line width=3.75]    (130,350) -- (190,410) ;

		\draw [line width=3.75]    (130,70) -- (130,350) ;

		\draw [line width=3.75]    (530,70) -- (530,350) ;

		\draw [line width=3.75]    (470,10) -- (530,70) ;

		\draw [line width=3.75]    (470,410) -- (530,350) ;

		\draw [line width=3.75]    (130,70) -- (169.75,30.25) -- (190,10) ;

		\draw [line width=3.75]    (190,210) -- (146.84,109.28) -- (130,70) ;

		\draw [line width=3.75]    (190,210) -- (167.27,263.04) -- (130,350) ;

		\draw  [draw opacity=0][fill={rgb, 255:red, 0; green, 0; blue, 0 }  ,fill opacity=0.25 ] (190,210) -- (130,350) -- (130,70) -- cycle ;
		\draw [line width=3.75]    (470,210) -- (530,350) ;

		\draw [line width=3.75]    (470,210) -- (530,70) ;

		\draw  [draw opacity=0][fill={rgb, 255:red, 0; green, 0; blue, 0 }  ,fill opacity=0.25 ] (470,210) -- (530,70) -- (530,350) -- cycle ;
		
		\draw [line width=3.75]    (190,210) -- (280,210) -- (470,210) ;

		\draw [line width=3.75]    (190,210) -- (190,10) ;

		\draw [line width=3.75]    (190,210) -- (470,10) ;

		\draw  [draw opacity=0][fill={rgb, 255:red, 0; green, 0; blue, 0 }  ,fill opacity=0.25 ] (470,10) -- (190,210) -- (190,10) -- cycle ;
		\draw  [draw opacity=0][fill={rgb, 255:red, 0; green, 0; blue, 0 }  ,fill opacity=0.25 ] (190,10) -- (130,70) -- (130,10) -- cycle ;
		\draw  [draw opacity=0][fill={rgb, 255:red, 0; green, 0; blue, 0 }  ,fill opacity=0.25 ] (130,350) -- (190,410) -- (130,410) -- cycle ;
		\draw  [draw opacity=0][fill={rgb, 255:red, 0; green, 0; blue, 0 }  ,fill opacity=0.25 ] (530,70) -- (470,10) -- (530,10) -- cycle ;
		\draw  [draw opacity=0][fill={rgb, 255:red, 0; green, 0; blue, 0 }  ,fill opacity=0.25 ] (470,410) -- (530,350) -- (530,410) -- cycle ;
		\draw [line width=3.75]    (330,350) -- (190,410) ;

		\draw [line width=3.75]    (330,350) -- (470,410) ;

		\draw  [draw opacity=0][fill={rgb, 255:red, 0; green, 0; blue, 0 }  ,fill opacity=0.25 ] (330,350) -- (470,410) -- (190,410) -- cycle ;
		
		\draw [line width=3.75]    (190,210) -- (330,350) ;

		\draw [line width=3.75]    (190,410) -- (470,410) ;

		\draw [line width=3.75]    (190,10) -- (470,10) ;

		
		\draw (130,350) node[VertexStyle, font={\bfseries\Huge}] (0) {a} ;
		
		\draw (190,410) node[VertexStyle, font={\bfseries\Huge}] (0) {b} ;
		
		\draw (130,70) node[VertexStyle, font={\bfseries\Huge}] (0) {c} ;
		
		\draw (530,70) node[VertexStyle, font={\bfseries\Huge}] (0) {c} ;
		
		\draw (530,350) node[VertexStyle, font={\bfseries\Huge}] (0) {a} ;
		
		\draw (470,10) node[VertexStyle, font={\bfseries\Huge}] (0) {d} ;
		
		\draw (470,410) node[VertexStyle, font={\bfseries\Huge}] (0) {d} ;
		
		\draw (190,10) node[VertexStyle, font={\bfseries\Huge}] (0) {b} ;
		
		\draw (190,210) node[VertexStyle, font={\bfseries\Huge}] (0) {e} ;
		
		\draw (470,210) node[VertexStyle, font={\bfseries\Huge}] (0) {f} ;
		
		\draw (330,350) node[VertexStyle, font={\bfseries\Huge}] (0) {g} ;

		\end{tikzpicture}
	}

\captionof{figure}{}\label{T4}}

\columnbreak

\center{
\scalebox{0.35}{
	
	\begin{tikzpicture}[x=0.75pt,y=0.75pt,yscale=-1,xscale=1]
	
%
		\draw   (130,10) -- (530,10) -- (530,410) -- (130,410) -- cycle ;
		\draw  [color={rgb, 255:red, 0; green, 0; blue, 0 }  ,draw opacity=1 ][fill={rgb, 255:red, 0; green, 0; blue, 0 }  ,fill opacity=1 ] (120,230) -- (130,210) -- (140,230) -- (130,220) -- cycle ;
		\draw  [color={rgb, 255:red, 0; green, 0; blue, 0 }  ,draw opacity=1 ][fill={rgb, 255:red, 0; green, 0; blue, 0 }  ,fill opacity=1 ] (520,230) -- (530,210) -- (540,230) -- (530,220) -- cycle ;
		\draw  [color={rgb, 255:red, 0; green, 0; blue, 0 }  ,draw opacity=1 ][fill={rgb, 255:red, 0; green, 0; blue, 0 }  ,fill opacity=1 ] (310,0) -- (330,10) -- (310,20) -- (320,10) -- cycle ;
		\draw  [color={rgb, 255:red, 0; green, 0; blue, 0 }  ,draw opacity=1 ][fill={rgb, 255:red, 0; green, 0; blue, 0 }  ,fill opacity=1 ] (290,0) -- (310,10) -- (290,20) -- (300,10) -- cycle ;
		\draw  [color={rgb, 255:red, 0; green, 0; blue, 0 }  ,draw opacity=1 ][fill={rgb, 255:red, 0; green, 0; blue, 0 }  ,fill opacity=1 ] (310,400) -- (330,410) -- (310,420) -- (320,410) -- cycle ;
		\draw  [color={rgb, 255:red, 0; green, 0; blue, 0 }  ,draw opacity=1 ][fill={rgb, 255:red, 0; green, 0; blue, 0 }  ,fill opacity=1 ] (290,400) -- (310,410) -- (290,420) -- (300,410) -- cycle ;
	\draw [line width=3.75]    (130,350) -- (190,410) ;

	\draw [line width=3.75]    (130,70) -- (130,350) ;

	\draw [line width=3.75]    (530,70) -- (530,350) ;

	\draw [line width=3.75]    (470,10) -- (530,70) ;

	\draw [line width=3.75]    (470,410) -- (530,350) ;

	\draw [line width=3.75]    (130,70) -- (169.75,30.25) -- (190,10) ;

	\draw [line width=3.75]    (190,210) -- (146.84,109.28) -- (130,70) ;

	\draw [line width=3.75]    (190,210) -- (167.27,263.04) -- (130,350) ;

	\draw  [draw opacity=0][fill={rgb, 255:red, 0; green, 0; blue, 0 }  ,fill opacity=0.25 ] (190,210) -- (130,350) -- (130,70) -- cycle ;
	\draw [line width=3.75]    (470,210) -- (530,350) ;

	\draw [line width=3.75]    (470,210) -- (530,70) ;

	\draw  [draw opacity=0][fill={rgb, 255:red, 0; green, 0; blue, 0 }  ,fill opacity=0.25 ] (470,210) -- (530,70) -- (530,350) -- cycle ;
	
	\draw [line width=3.75]    (190,210) -- (280,210) -- (470,210) ;

	\draw [line width=3.75]    (190,210) -- (190,10) ;

	\draw [line width=3.75]    (190,210) -- (470,10) ;

	\draw  [draw opacity=0][fill={rgb, 255:red, 0; green, 0; blue, 0 }  ,fill opacity=0.25 ] (470,10) -- (190,210) -- (190,10) -- cycle ;
	\draw  [draw opacity=0][fill={rgb, 255:red, 0; green, 0; blue, 0 }  ,fill opacity=0.25 ] (190,10) -- (130,70) -- (130,10) -- cycle ;
	\draw  [draw opacity=0][fill={rgb, 255:red, 0; green, 0; blue, 0 }  ,fill opacity=0.25 ] (130,350) -- (190,410) -- (130,410) -- cycle ;
	\draw  [draw opacity=0][fill={rgb, 255:red, 0; green, 0; blue, 0 }  ,fill opacity=0.25 ] (530,70) -- (470,10) -- (530,10) -- cycle ;
	\draw  [draw opacity=0][fill={rgb, 255:red, 0; green, 0; blue, 0 }  ,fill opacity=0.25 ] (470,410) -- (530,350) -- (530,410) -- cycle ;
	\draw [line width=3.75]    (330,350) -- (190,410) ;

	\draw [line width=3.75]    (330,350) -- (470,410) ;

	\draw  [draw opacity=0][fill={rgb, 255:red, 0; green, 0; blue, 0 }  ,fill opacity=0.25 ] (330,350) -- (470,410) -- (190,410) -- cycle ;
	
	\draw [line width=3.75]    (190,210) -- (330,350) ;

	\draw [line width=3.75]    (130,350) -- (330,350) ;

	\draw [line width=3.75]    (470,10) -- (470,210) ;

	\draw [line width=3.75]    (190,10) -- (470,10) ;

	\draw [line width=3.75]    (190,410) -- (470,410) ;

	\draw  [draw opacity=0][fill={rgb, 255:red, 0; green, 0; blue, 0 }  ,fill opacity=0.25 ] (190,210) -- (470,10) -- (470,210) -- cycle ;
	\draw  [draw opacity=0][fill={rgb, 255:red, 0; green, 0; blue, 0 }  ,fill opacity=0.25 ] (190.5,210) -- (330,350) -- (130,350) -- cycle ;
	
	\draw (130,350) node[VertexStyle, font={\bfseries\Huge}] (0) {a} ;
	
	\draw (190,410) node[VertexStyle, font={\bfseries\Huge}] (0) {b} ;
	
	\draw (130,70) node[VertexStyle, font={\bfseries\Huge}] (0) {c} ;
	
	\draw (530,70) node[VertexStyle, font={\bfseries\Huge}] (0) {c} ;
	
	\draw (530,350) node[VertexStyle, font={\bfseries\Huge}] (0) {a} ;
	
	\draw (470,10) node[VertexStyle, font={\bfseries\Huge}] (0) {d} ;
	
	\draw (470,410) node[VertexStyle, font={\bfseries\Huge}] (0) {d} ;
	
	\draw (190,10) node[VertexStyle, font={\bfseries\Huge}] (0) {b} ;
	
	\draw (190,210) node[VertexStyle, font={\bfseries\Huge}] (0) {e} ;
	
	\draw (470,210) node[VertexStyle, font={\bfseries\Huge}] (0) {f} ;
	
	\draw (330,350) node[VertexStyle, font={\bfseries\Huge}] (0) {g} ;

	\end{tikzpicture}
}

\captionof{figure}{}\label{T5}}

\columnbreak
	
	\center{
	\scalebox{0.35}{
		
		\begin{tikzpicture}[x=0.75pt,y=0.75pt,yscale=-1,xscale=1]
%
%
		\draw   (130,10) -- (530,10) -- (530,410) -- (130,410) -- cycle ;
		\draw  [color={rgb, 255:red, 0; green, 0; blue, 0 }  ,draw opacity=1 ][fill={rgb, 255:red, 0; green, 0; blue, 0 }  ,fill opacity=1 ] (120,230) -- (130,210) -- (140,230) -- (130,220) -- cycle ;
		\draw  [color={rgb, 255:red, 0; green, 0; blue, 0 }  ,draw opacity=1 ][fill={rgb, 255:red, 0; green, 0; blue, 0 }  ,fill opacity=1 ] (520,230) -- (530,210) -- (540,230) -- (530,220) -- cycle ;
		\draw  [color={rgb, 255:red, 0; green, 0; blue, 0 }  ,draw opacity=1 ][fill={rgb, 255:red, 0; green, 0; blue, 0 }  ,fill opacity=1 ] (310,0) -- (330,10) -- (310,20) -- (320,10) -- cycle ;
		\draw  [color={rgb, 255:red, 0; green, 0; blue, 0 }  ,draw opacity=1 ][fill={rgb, 255:red, 0; green, 0; blue, 0 }  ,fill opacity=1 ] (290,0) -- (310,10) -- (290,20) -- (300,10) -- cycle ;
		\draw  [color={rgb, 255:red, 0; green, 0; blue, 0 }  ,draw opacity=1 ][fill={rgb, 255:red, 0; green, 0; blue, 0 }  ,fill opacity=1 ] (310,400) -- (330,410) -- (310,420) -- (320,410) -- cycle ;
		\draw  [color={rgb, 255:red, 0; green, 0; blue, 0 }  ,draw opacity=1 ][fill={rgb, 255:red, 0; green, 0; blue, 0 }  ,fill opacity=1 ] (290,400) -- (310,410) -- (290,420) -- (300,410) -- cycle ;
		\draw [line width=3.75]    (130,350) -- (190,410) ;

		\draw [line width=3.75]    (130,70) -- (130,350) ;

		\draw [line width=3.75]    (530,70) -- (530,350) ;

		\draw [line width=3.75]    (470,10) -- (530,70) ;

		\draw [line width=3.75]    (470,410) -- (530,350) ;

		\draw [line width=3.75]    (130,70) -- (169.75,30.25) -- (190,10) ;

		\draw [line width=3.75]    (190,210) -- (146.84,109.28) -- (130,70) ;

		\draw [line width=3.75]    (190,210) -- (167.27,263.04) -- (130,350) ;

		\draw  [draw opacity=0][fill={rgb, 255:red, 0; green, 0; blue, 0 }  ,fill opacity=0.25 ] (190,210) -- (130,350) -- (130,70) -- cycle ;
		\draw [line width=3.75]    (470,210) -- (530,350) ;

		\draw [line width=3.75]    (470,210) -- (530,70) ;

		\draw  [draw opacity=0][fill={rgb, 255:red, 0; green, 0; blue, 0 }  ,fill opacity=0.25 ] (470,210) -- (530,70) -- (530,350) -- cycle ;
		
		\draw [line width=3.75]    (190,210) -- (280,210) -- (470,210) ;

		\draw [line width=3.75]    (190,210) -- (190,10) ;

		\draw [line width=3.75]    (190,210) -- (470,10) ;

		\draw  [draw opacity=0][fill={rgb, 255:red, 0; green, 0; blue, 0 }  ,fill opacity=0.25 ] (470,10) -- (190,210) -- (190,10) -- cycle ;
		\draw  [draw opacity=0][fill={rgb, 255:red, 0; green, 0; blue, 0 }  ,fill opacity=0.25 ] (190,10) -- (130,70) -- (130,10) -- cycle ;
		\draw  [draw opacity=0][fill={rgb, 255:red, 0; green, 0; blue, 0 }  ,fill opacity=0.25 ] (130,350) -- (190,410) -- (130,410) -- cycle ;
		\draw  [draw opacity=0][fill={rgb, 255:red, 0; green, 0; blue, 0 }  ,fill opacity=0.25 ] (530,70) -- (470,10) -- (530,10) -- cycle ;
		\draw  [draw opacity=0][fill={rgb, 255:red, 0; green, 0; blue, 0 }  ,fill opacity=0.25 ] (470,410) -- (530,350) -- (530,410) -- cycle ;
		\draw [line width=3.75]    (330,350) -- (190,410) ;

		\draw [line width=3.75]    (330,350) -- (470,410) ;

		\draw  [draw opacity=0][fill={rgb, 255:red, 0; green, 0; blue, 0 }  ,fill opacity=0.25 ] (330,350) -- (470,410) -- (190,410) -- cycle ;
		
		\draw [line width=3.75]    (190,210) -- (330,350) ;

		\draw [line width=3.75]    (130,350) -- (330,350) ;

		\draw [line width=3.75]    (470,10) -- (470,210) ;

		\draw [line width=3.75]    (190,10) -- (470,10) ;

		\draw [line width=3.75]    (190,410) -- (470,410) ;

		\draw  [draw opacity=0][fill={rgb, 255:red, 0; green, 0; blue, 0 }  ,fill opacity=0.25 ] (190,210) -- (470,10) -- (470,210) -- cycle ;
		\draw  [draw opacity=0][fill={rgb, 255:red, 0; green, 0; blue, 0 }  ,fill opacity=0.25 ] (190.5,210) -- (330,350) -- (130,350) -- cycle ;
		\draw [line width=3.75]    (470,210) -- (400,280) ;

		\draw [line width=3.75]    (400,280) -- (190.5,210) ;

		\draw [line width=3.75]    (400,280) -- (330,350) ;

		\draw [line width=3.75]    (530,350) -- (400,280) ;

		\draw [line width=3.75]    (470,410) -- (400,280) ;

		\draw  [draw opacity=0][fill={rgb, 255:red, 0; green, 0; blue, 0 }  ,fill opacity=0.25 ] (400.5,280) -- (190,210) -- (470,210) -- cycle ;
		\draw  [draw opacity=0][fill={rgb, 255:red, 0; green, 0; blue, 0 }  ,fill opacity=0.25 ] (400,280) -- (330,350) -- (190,210) -- cycle ;
		
		\draw (130,350) node[VertexStyle, font={\bfseries\Huge}] (0) {a} ;
		
		\draw (190,410) node[VertexStyle, font={\bfseries\Huge}] (0) {b} ;
		
		\draw (130,70) node[VertexStyle, font={\bfseries\Huge}] (0) {c} ;
		
		\draw (530,70) node[VertexStyle, font={\bfseries\Huge}] (0) {c} ;
		
		\draw (530,350) node[VertexStyle, font={\bfseries\Huge}] (0) {a} ;
		
		\draw (470,10) node[VertexStyle, font={\bfseries\Huge}] (0) {d} ;
		
		\draw (470,410) node[VertexStyle, font={\bfseries\Huge}] (0) {d} ;
		
		\draw (190,10) node[VertexStyle, font={\bfseries\Huge}] (0) {b} ;
		
		\draw (190,210) node[VertexStyle, font={\bfseries\Huge}] (0) {e} ;
		
		\draw (470,210) node[VertexStyle, font={\bfseries\Huge}] (0) {f} ;
		
		\draw (330,350) node[VertexStyle, font={\bfseries\Huge}] (0) {g} ;
		
		\draw (400,280) node[VertexStyle, font={\bfseries\Huge}] (0) {h} ;

		\end{tikzpicture}
		
	}

\captionof{figure}{}\label{T6}}

\end{multicols}

Next the Klein bottle.

\begin{proof}[Proof of Theorem \ref{klein}]
	Suppose that $G$ is an Euler impure graph on the Klein bottle. Then by Lemma \ref{4-face} there exists an embedding of $G$ with a 4-face whose vertices induce $K_4$. There are two such embedding of $K_4$ (see Figures \ref{K1} and \ref{K2}). We will first consider the case when the vertices of this 4-face inducing $K_4$ embed as in Figure \ref{K2}.
	
	As $K_6$ is embeddable, the 2-cell $C$ 
	bounded by $acbdacdb$ must contain at least 3 additional vertices in its interior. Both edges $ac$ and $bd$ can be flipped into the 4-face, so we may apply Lemma \ref{T+K lemma} and without loss of generality deduce that there is an embedding in which $C$ contains vertices $e,f,g$ such that $e$ is adjacent to both $f$ and $g$ and $ace$, $caf$, $bdx$ and $dby$ are all faces (see Figure \ref{K3}).
	
	By possibly flipping the edge $df$ (if it is an edge), we may assume that $f$ is adjacent to some vertex $x$ appearing first clockwise after $e$ but before $c$. But as the edge $ac$ and then $ef$ are flippable, we see that we must have that $x=d$ and so further that $fed$ is a face.
	
	As edges $bd$ and then $eg$ are flippable, we see that $aeg$ lies on a face. Furthermore as $b$ and $f$ are non-adjacent we must also have that  $a$ is adjacent to $g$, with $aeg$ being a face (see Figure \ref{K4}).
	
	Let $x$ be the first vertex adjacent to $g$ appearing clockwise after $e$. Then as edges $bd$ and then $eg$ are flippable and $a$ does not lie on $gefcd$, we see that $x$ must be a vertex lying on $gefcd$. As there is already an edge between $e$ and every vertex on $gefcd$ in the embedding, we see that $gef$ lies on a face and so $x=f$ with $gef$ being a face.
	
	Now considering the sequence of edge flips $ac$, $ef$, $dg$, we see that $cdg$ must lie on a face. Furthermore as $b$ and $f$ are non-adjacent, we see that $g$ and $c$ must be adjacent and so $cdg$ is in fact a face (see Figure \ref{K5}).
	
	By the sequence of edge flips $bd$, $eg$, $af$, we see that $daf$ must be a face. Furthermore by flipping $bd$, $eg$, $af$, $cd$, $gb$, we get that $bag$ is a face too. Flipping edges $ac$, $ef$, $dg$ and then $bc$, we see that $cbe$ must also be a face. Going further once again and considering the sequence of flips $ac$, $ef$, $dg$, $bc$, $ae$, $gf$, we see that $gfc$ is a face.
	
	Hence in this case $G\cong K_7-e$ (with $b$ and $f$ being the pair of non-adjacent vertices) which we know to be Euler impure \cite{franklin1934six}. It remains now to consider the case when $G$ embeds with a 4-face inducing $K_4$ as in Figure \ref{K1}.
	
	As $bd$ is flippable, we may assume without loss of generality that $dbac$ is a face. So $a$ must be adjacent to at least one vertex appearing clockwise between $c$ and $b$. Let $e$ be the first neighbour of $a$ appearing clockwise after $c$. Then by possibly flipping $ec$, we see that $ace$ is a face. Furthermore $e$ must be adjacent to both $b$ and $d$ as $ac$ is flippable.
	
	Note that the vertices $a,b,c,d,e$ must have degree at least 5 as we could flip $ac$ and then both $eb$ and $ed$ to get a new embedding (as in Figure \ref{K7}), in which $aecdb$ bounds a mobius strip that must contain at least one additional vertex.
	
	Hence $bae$ and $ecd$ are not faces and so $G[\{a,b,c,d,e\}]$ must embed as in Figure \ref{K6}.
	
	Let $f$ be the first vertex adjacent to $a$ appearing anticlockwise after $b$ (we know that $f\not=e$ as $a$ has degree at least 5). Then $baf$ must be a face. Also as $ab$ is flippable, $f$ must also be adjacent to $b$ and $d$ with $dbf$ being a face.
	
	After flipping $ac$ both edges $eb$ and $ed$ can be flipped. Similarly after flipping $bd$ both edges $ea$ and $ec$ can be flipped. So by possibly flipping $ef$ we see that $ecfb$, $feb$, $efa$ are all faces (see Figure \ref{K8}). But then $G\cong K_6$ a contradiction.
\end{proof}

\begin{multicols}{3}
	\center{
	\scalebox{0.35}{
		
		\begin{tikzpicture}[x=0.75pt,y=0.75pt,yscale=-1,xscale=1]
		
		
		\draw   (130,10) -- (530,10) -- (530,410) -- (130,410) -- cycle ;
		\draw  [color={rgb, 255:red, 0; green, 0; blue, 0 }  ,draw opacity=1 ][fill={rgb, 255:red, 0; green, 0; blue, 0 }  ,fill opacity=1 ] (120,230) -- (130,210) -- (140,230) -- (130,220) -- cycle ;
		\draw  [color={rgb, 255:red, 0; green, 0; blue, 0 }  ,draw opacity=1 ][fill={rgb, 255:red, 0; green, 0; blue, 0 }  ,fill opacity=1 ] (540,190) -- (530,210) -- (520,190) -- (530,200) -- cycle ;
		\draw  [color={rgb, 255:red, 0; green, 0; blue, 0 }  ,draw opacity=1 ][fill={rgb, 255:red, 0; green, 0; blue, 0 }  ,fill opacity=1 ] (310,0) -- (330,10) -- (310,20) -- (320,10) -- cycle ;
		\draw  [color={rgb, 255:red, 0; green, 0; blue, 0 }  ,draw opacity=1 ][fill={rgb, 255:red, 0; green, 0; blue, 0 }  ,fill opacity=1 ] (290,0) -- (310,10) -- (290,20) -- (300,10) -- cycle ;
		\draw  [color={rgb, 255:red, 0; green, 0; blue, 0 }  ,draw opacity=1 ][fill={rgb, 255:red, 0; green, 0; blue, 0 }  ,fill opacity=1 ] (310,400) -- (330,410) -- (310,420) -- (320,410) -- cycle ;
		\draw  [color={rgb, 255:red, 0; green, 0; blue, 0 }  ,draw opacity=1 ][fill={rgb, 255:red, 0; green, 0; blue, 0 }  ,fill opacity=1 ] (290,400) -- (310,410) -- (290,420) -- (300,410) -- cycle ;
		\draw [line width=3.75]    (130,350) -- (190,410) ;

		\draw [line width=3.75]    (130,70) -- (130,350) ;

		\draw [line width=3.75]    (530,70) -- (530,350) ;

		\draw [line width=3.75]    (470,10) -- (530,70) ;

		\draw [line width=3.75]    (470,410) -- (530,350) ;

		\draw [line width=3.75]    (130,70) -- (169.75,30.25) -- (190,10) ;

		\draw [line width=3.75]    (190,210) -- (146.84,109.28) -- (130,70) ;

		\draw [line width=3.75]    (190,210) -- (167.27,263.04) -- (130,350) ;

		\draw  [draw opacity=0][fill={rgb, 255:red, 0; green, 0; blue, 0 }  ,fill opacity=0.25 ] (190,210) -- (130,350) -- (130,70) -- cycle ;
		\draw [line width=3.75]    (470,210) -- (530,350) ;

		\draw [line width=3.75]    (470,210) -- (530,70) ;

		\draw  [draw opacity=0][fill={rgb, 255:red, 0; green, 0; blue, 0 }  ,fill opacity=0.25 ] (470,210) -- (530,70) -- (530,350) -- cycle ;
		
		\draw [line width=3.75]    (190,210) -- (280,210) -- (470,210) ;

		\draw [line width=3.75]    (190,210) -- (190,10) ;

		\draw [line width=3.75]    (190,210) -- (470,10) ;

		\draw  [draw opacity=0][fill={rgb, 255:red, 0; green, 0; blue, 0 }  ,fill opacity=0.25 ] (470,10) -- (190,210) -- (190,10) -- cycle ;
		\draw  [draw opacity=0][fill={rgb, 255:red, 0; green, 0; blue, 0 }  ,fill opacity=0.25 ] (190,10) -- (130,70) -- (130,10) -- cycle ;
		\draw  [draw opacity=0][fill={rgb, 255:red, 0; green, 0; blue, 0 }  ,fill opacity=0.25 ] (130,350) -- (190,410) -- (130,410) -- cycle ;
		\draw  [draw opacity=0][fill={rgb, 255:red, 0; green, 0; blue, 0 }  ,fill opacity=0.25 ] (530,70) -- (470,10) -- (530,10) -- cycle ;
		\draw  [draw opacity=0][fill={rgb, 255:red, 0; green, 0; blue, 0 }  ,fill opacity=0.25 ] (470,410) -- (530,350) -- (530,410) -- cycle ;
		\draw [line width=3.75]    (330,350) -- (190,410) ;

		\draw [line width=3.75]    (330,350) -- (470,410) ;

		\draw  [draw opacity=0][fill={rgb, 255:red, 0; green, 0; blue, 0 }  ,fill opacity=0.25 ] (330,350) -- (470,410) -- (190,410) -- cycle ;
		
		\draw [line width=3.75]    (190,210) -- (330,350) ;

		\draw [line width=3.75]    (190,410) -- (470,410) ;

		\draw [line width=3.75]    (190,10) -- (470,10) ;

		
		\draw (130,350) node[VertexStyle, font={\bfseries\Huge}] (0) {a} ;
		
		\draw (190,410) node[VertexStyle, font={\bfseries\Huge}] (0) {b} ;
		
		\draw (130,70) node[VertexStyle, font={\bfseries\Huge}] (0) {c} ;
		
		\draw (530,70) node[VertexStyle, font={\bfseries\Huge}] (0) {a} ;
		
		\draw (530,350) node[VertexStyle, font={\bfseries\Huge}] (0) {c} ;
		
		\draw (470,10) node[VertexStyle, font={\bfseries\Huge}] (0) {d} ;
		
		\draw (470,410) node[VertexStyle, font={\bfseries\Huge}] (0) {d} ;
		
		\draw (190,10) node[VertexStyle, font={\bfseries\Huge}] (0) {b} ;
		
		\draw (190,210) node[VertexStyle, font={\bfseries\Huge}] (0) {e} ;
		
		\draw (470,210) node[VertexStyle, font={\bfseries\Huge}] (0) {f} ;
		
		\draw (330,350) node[VertexStyle, font={\bfseries\Huge}] (0) {g} ;

		\end{tikzpicture}
		
	}
	
	\captionof{figure}{}\label{K3}}

	\columnbreak
	
	\center{
	\scalebox{0.35}{
		
		\begin{tikzpicture}[x=0.75pt,y=0.75pt,yscale=-1,xscale=1]
		
		
		\draw   (130,10) -- (530,10) -- (530,410) -- (130,410) -- cycle ;
		\draw  [color={rgb, 255:red, 0; green, 0; blue, 0 }  ,draw opacity=1 ][fill={rgb, 255:red, 0; green, 0; blue, 0 }  ,fill opacity=1 ] (120,230) -- (130,210) -- (140,230) -- (130,220) -- cycle ;
		\draw  [color={rgb, 255:red, 0; green, 0; blue, 0 }  ,draw opacity=1 ][fill={rgb, 255:red, 0; green, 0; blue, 0 }  ,fill opacity=1 ] (540,190) -- (530,210) -- (520,190) -- (530,200) -- cycle ;
		\draw  [color={rgb, 255:red, 0; green, 0; blue, 0 }  ,draw opacity=1 ][fill={rgb, 255:red, 0; green, 0; blue, 0 }  ,fill opacity=1 ] (310,0) -- (330,10) -- (310,20) -- (320,10) -- cycle ;
		\draw  [color={rgb, 255:red, 0; green, 0; blue, 0 }  ,draw opacity=1 ][fill={rgb, 255:red, 0; green, 0; blue, 0 }  ,fill opacity=1 ] (290,0) -- (310,10) -- (290,20) -- (300,10) -- cycle ;
		\draw  [color={rgb, 255:red, 0; green, 0; blue, 0 }  ,draw opacity=1 ][fill={rgb, 255:red, 0; green, 0; blue, 0 }  ,fill opacity=1 ] (310,400) -- (330,410) -- (310,420) -- (320,410) -- cycle ;
		\draw  [color={rgb, 255:red, 0; green, 0; blue, 0 }  ,draw opacity=1 ][fill={rgb, 255:red, 0; green, 0; blue, 0 }  ,fill opacity=1 ] (290,400) -- (310,410) -- (290,420) -- (300,410) -- cycle ;
		\draw [line width=3.75]    (130,350) -- (190,410) ;

		\draw [line width=3.75]    (130,70) -- (130,350) ;

		\draw [line width=3.75]    (530,70) -- (530,350) ;

		\draw [line width=3.75]    (470,10) -- (530,70) ;

		\draw [line width=3.75]    (470,410) -- (530,350) ;

		\draw [line width=3.75]    (130,70) -- (169.75,30.25) -- (190,10) ;

		\draw [line width=3.75]    (190,210) -- (146.84,109.28) -- (130,70) ;

		\draw [line width=3.75]    (190,210) -- (167.27,263.04) -- (130,350) ;

		\draw  [draw opacity=0][fill={rgb, 255:red, 0; green, 0; blue, 0 }  ,fill opacity=0.25 ] (190,210) -- (130,350) -- (130,70) -- cycle ;
		\draw [line width=3.75]    (470,210) -- (530,350) ;

		\draw [line width=3.75]    (470,210) -- (530,70) ;

		\draw  [draw opacity=0][fill={rgb, 255:red, 0; green, 0; blue, 0 }  ,fill opacity=0.25 ] (470,210) -- (530,70) -- (530,350) -- cycle ;
		
		\draw [line width=3.75]    (190,210) -- (280,210) -- (470,210) ;

		\draw [line width=3.75]    (190,210) -- (190,10) ;

		\draw [line width=3.75]    (190,210) -- (470,10) ;

		\draw  [draw opacity=0][fill={rgb, 255:red, 0; green, 0; blue, 0 }  ,fill opacity=0.25 ] (470,10) -- (190,210) -- (190,10) -- cycle ;
		\draw  [draw opacity=0][fill={rgb, 255:red, 0; green, 0; blue, 0 }  ,fill opacity=0.25 ] (190,10) -- (130,70) -- (130,10) -- cycle ;
		\draw  [draw opacity=0][fill={rgb, 255:red, 0; green, 0; blue, 0 }  ,fill opacity=0.25 ] (130,350) -- (190,410) -- (130,410) -- cycle ;
		\draw  [draw opacity=0][fill={rgb, 255:red, 0; green, 0; blue, 0 }  ,fill opacity=0.25 ] (530,70) -- (470,10) -- (530,10) -- cycle ;
		\draw  [draw opacity=0][fill={rgb, 255:red, 0; green, 0; blue, 0 }  ,fill opacity=0.25 ] (470,410) -- (530,350) -- (530,410) -- cycle ;
		\draw [line width=3.75]    (330,350) -- (190,410) ;

		\draw [line width=3.75]    (330,350) -- (470,410) ;

		\draw  [draw opacity=0][fill={rgb, 255:red, 0; green, 0; blue, 0 }  ,fill opacity=0.25 ] (330,350) -- (470,410) -- (190,410) -- cycle ;
		
		\draw [line width=3.75]    (190,210) -- (330,350) ;

		\draw [line width=3.75]    (130,350) -- (330,350) ;

		\draw [line width=3.75]    (470,10) -- (470,210) ;

		\draw [line width=3.75]    (190,10) -- (470,10) ;

		\draw [line width=3.75]    (190,410) -- (470,410) ;

		\draw  [draw opacity=0][fill={rgb, 255:red, 0; green, 0; blue, 0 }  ,fill opacity=0.25 ] (190,210) -- (470,10) -- (470,210) -- cycle ;
		\draw  [draw opacity=0][fill={rgb, 255:red, 0; green, 0; blue, 0 }  ,fill opacity=0.25 ] (190.5,210) -- (330,350) -- (130,350) -- cycle ;
		
		\draw (130,350) node[VertexStyle, font={\bfseries\Huge}] (0) {a} ;
		
		\draw (190,410) node[VertexStyle, font={\bfseries\Huge}] (0) {b} ;
		
		\draw (130,70) node[VertexStyle, font={\bfseries\Huge}] (0) {c} ;
		
		\draw (530,70) node[VertexStyle, font={\bfseries\Huge}] (0) {a} ;
		
		\draw (530,350) node[VertexStyle, font={\bfseries\Huge}] (0) {c} ;
		
		\draw (470,10) node[VertexStyle, font={\bfseries\Huge}] (0) {d} ;
		
		\draw (470,410) node[VertexStyle, font={\bfseries\Huge}] (0) {d} ;
		
		\draw (190,10) node[VertexStyle, font={\bfseries\Huge}] (0) {b} ;
		
		\draw (190,210) node[VertexStyle, font={\bfseries\Huge}] (0) {e} ;
		
		\draw (470,210) node[VertexStyle, font={\bfseries\Huge}] (0) {f} ;
		
		\draw (330,350) node[VertexStyle, font={\bfseries\Huge}] (0) {g} ;

		\end{tikzpicture}

	}
	
	\captionof{figure}{}\label{K4}}

	\columnbreak
	
	\center{
	\scalebox{0.35}{
		
		\begin{tikzpicture}[x=0.75pt,y=0.75pt,yscale=-1,xscale=1]
		
		
		\draw   (130,10) -- (530,10) -- (530,410) -- (130,410) -- cycle ;
		\draw  [color={rgb, 255:red, 0; green, 0; blue, 0 }  ,draw opacity=1 ][fill={rgb, 255:red, 0; green, 0; blue, 0 }  ,fill opacity=1 ] (120,230) -- (130,210) -- (140,230) -- (130,220) -- cycle ;
		\draw  [color={rgb, 255:red, 0; green, 0; blue, 0 }  ,draw opacity=1 ][fill={rgb, 255:red, 0; green, 0; blue, 0 }  ,fill opacity=1 ] (540,190) -- (530,210) -- (520,190) -- (530,200) -- cycle ;
		\draw  [color={rgb, 255:red, 0; green, 0; blue, 0 }  ,draw opacity=1 ][fill={rgb, 255:red, 0; green, 0; blue, 0 }  ,fill opacity=1 ] (310,0) -- (330,10) -- (310,20) -- (320,10) -- cycle ;
		\draw  [color={rgb, 255:red, 0; green, 0; blue, 0 }  ,draw opacity=1 ][fill={rgb, 255:red, 0; green, 0; blue, 0 }  ,fill opacity=1 ] (290,0) -- (310,10) -- (290,20) -- (300,10) -- cycle ;
		\draw  [color={rgb, 255:red, 0; green, 0; blue, 0 }  ,draw opacity=1 ][fill={rgb, 255:red, 0; green, 0; blue, 0 }  ,fill opacity=1 ] (310,400) -- (330,410) -- (310,420) -- (320,410) -- cycle ;
		\draw  [color={rgb, 255:red, 0; green, 0; blue, 0 }  ,draw opacity=1 ][fill={rgb, 255:red, 0; green, 0; blue, 0 }  ,fill opacity=1 ] (290,400) -- (310,410) -- (290,420) -- (300,410) -- cycle ;
		\draw [line width=3.75]    (130,350) -- (190,410) ;

		\draw [line width=3.75]    (130,70) -- (130,350) ;

		\draw [line width=3.75]    (530,70) -- (530,350) ;

		\draw [line width=3.75]    (470,10) -- (530,70) ;

		\draw [line width=3.75]    (470,410) -- (530,350) ;

		\draw [line width=3.75]    (130,70) -- (169.75,30.25) -- (190,10) ;

		\draw [line width=3.75]    (190,210) -- (146.84,109.28) -- (130,70) ;

		\draw [line width=3.75]    (190,210) -- (167.27,263.04) -- (130,350) ;

		\draw  [draw opacity=0][fill={rgb, 255:red, 0; green, 0; blue, 0 }  ,fill opacity=0.25 ] (190,210) -- (130,350) -- (130,70) -- cycle ;
		\draw [line width=3.75]    (470,210) -- (530,350) ;

		\draw [line width=3.75]    (470,210) -- (530,70) ;

		\draw  [draw opacity=0][fill={rgb, 255:red, 0; green, 0; blue, 0 }  ,fill opacity=0.25 ] (470,210) -- (530,70) -- (530,350) -- cycle ;
		
		\draw [line width=3.75]    (190,210) -- (280,210) -- (470,210) ;

		\draw [line width=3.75]    (190,210) -- (190,10) ;

		\draw [line width=3.75]    (190,210) -- (470,10) ;

		\draw  [draw opacity=0][fill={rgb, 255:red, 0; green, 0; blue, 0 }  ,fill opacity=0.25 ] (470,10) -- (190,210) -- (190,10) -- cycle ;
		\draw  [draw opacity=0][fill={rgb, 255:red, 0; green, 0; blue, 0 }  ,fill opacity=0.25 ] (190,10) -- (130,70) -- (130,10) -- cycle ;
		\draw  [draw opacity=0][fill={rgb, 255:red, 0; green, 0; blue, 0 }  ,fill opacity=0.25 ] (130,350) -- (190,410) -- (130,410) -- cycle ;
		\draw  [draw opacity=0][fill={rgb, 255:red, 0; green, 0; blue, 0 }  ,fill opacity=0.25 ] (530,70) -- (470,10) -- (530,10) -- cycle ;
		\draw  [draw opacity=0][fill={rgb, 255:red, 0; green, 0; blue, 0 }  ,fill opacity=0.25 ] (470,410) -- (530,350) -- (530,410) -- cycle ;
		\draw [line width=3.75]    (330,350) -- (190,410) ;

		\draw [line width=3.75]    (330,350) -- (470,410) ;

		\draw  [draw opacity=0][fill={rgb, 255:red, 0; green, 0; blue, 0 }  ,fill opacity=0.25 ] (330,350) -- (470,410) -- (190,410) -- cycle ;
		
		\draw [line width=3.75]    (190,210) -- (330,350) ;

		\draw [line width=3.75]    (130,350) -- (330,350) ;

		\draw [line width=3.75]    (470,10) -- (470,210) ;

		\draw [line width=3.75]    (190,10) -- (470,10) ;

		\draw [line width=3.75]    (190,410) -- (470,410) ;

		\draw  [draw opacity=0][fill={rgb, 255:red, 0; green, 0; blue, 0 }  ,fill opacity=0.25 ] (190,210) -- (470,10) -- (470,210) -- cycle ;
		\draw  [draw opacity=0][fill={rgb, 255:red, 0; green, 0; blue, 0 }  ,fill opacity=0.25 ] (190.5,210) -- (330,350) -- (130,350) -- cycle ;
		\draw [line width=3.75]    (470,210) -- (330,350) ;

		\draw [line width=3.75]    (530,350) -- (330,350) ;

		\draw  [draw opacity=0][fill={rgb, 255:red, 0; green, 0; blue, 0 }  ,fill opacity=0.25 ] (469.5,410) -- (330,350) -- (530,350) -- cycle ;
		
		\draw (130,350) node[VertexStyle, font={\bfseries\Huge}] (0) {a} ;
		
		\draw (190,410) node[VertexStyle, font={\bfseries\Huge}] (0) {b} ;
		
		\draw (130,70) node[VertexStyle, font={\bfseries\Huge}] (0) {c} ;
		
		\draw (530,70) node[VertexStyle, font={\bfseries\Huge}] (0) {a} ;
		
		\draw (530,350) node[VertexStyle, font={\bfseries\Huge}] (0) {c} ;
		
		\draw (470,10) node[VertexStyle, font={\bfseries\Huge}] (0) {d} ;
		
		\draw (470,410) node[VertexStyle, font={\bfseries\Huge}] (0) {d} ;
		
		\draw (190,10) node[VertexStyle, font={\bfseries\Huge}] (0) {b} ;
		
		\draw (190,210) node[VertexStyle, font={\bfseries\Huge}] (0) {e} ;
		
		\draw (470,210) node[VertexStyle, font={\bfseries\Huge}] (0) {f} ;
		
		\draw (330,350) node[VertexStyle, font={\bfseries\Huge}] (0) {g} ;

		\end{tikzpicture}
		
	}
	
	\captionof{figure}{}\label{K5}}

\end{multicols}

\begin{multicols}{3}
	\center{
	\scalebox{0.35}{
		
		\begin{tikzpicture}[x=0.75pt,y=0.75pt,yscale=-1,xscale=1]
		
		
		\draw   (130,10) -- (530,10) -- (530,410) -- (130,410) -- cycle ;
		\draw  [color={rgb, 255:red, 0; green, 0; blue, 0 }  ,draw opacity=1 ][fill={rgb, 255:red, 0; green, 0; blue, 0 }  ,fill opacity=1 ] (120,230) -- (130,210) -- (140,230) -- (130,220) -- cycle ;
		\draw  [color={rgb, 255:red, 0; green, 0; blue, 0 }  ,draw opacity=1 ][fill={rgb, 255:red, 0; green, 0; blue, 0 }  ,fill opacity=1 ] (540,190) -- (530,210) -- (520,190) -- (530,200) -- cycle ;
		\draw  [color={rgb, 255:red, 0; green, 0; blue, 0 }  ,draw opacity=1 ][fill={rgb, 255:red, 0; green, 0; blue, 0 }  ,fill opacity=1 ] (310,0) -- (330,10) -- (310,20) -- (320,10) -- cycle ;
		\draw  [color={rgb, 255:red, 0; green, 0; blue, 0 }  ,draw opacity=1 ][fill={rgb, 255:red, 0; green, 0; blue, 0 }  ,fill opacity=1 ] (290,0) -- (310,10) -- (290,20) -- (300,10) -- cycle ;
		\draw  [color={rgb, 255:red, 0; green, 0; blue, 0 }  ,draw opacity=1 ][fill={rgb, 255:red, 0; green, 0; blue, 0 }  ,fill opacity=1 ] (310,400) -- (330,410) -- (310,420) -- (320,410) -- cycle ;
		\draw  [color={rgb, 255:red, 0; green, 0; blue, 0 }  ,draw opacity=1 ][fill={rgb, 255:red, 0; green, 0; blue, 0 }  ,fill opacity=1 ] (290,400) -- (310,410) -- (290,420) -- (300,410) -- cycle ;
		\draw [line width=3.75]    (190,350) -- (470,350) ;

		\draw [line width=3.75]    (130,350) -- (190,350) ;

		\draw [line width=3.75]    (190,350) -- (190,410) ;

		\draw  [draw opacity=0][fill={rgb, 255:red, 0; green, 0; blue, 0 }  ,fill opacity=0.25 ] (130,350) -- (190,350) -- (190,410) -- (130,410) -- cycle ;
		\draw [line width=3.75]    (470,350) -- (470,410) ;

		\draw [line width=3.75]    (470,350) -- (530,350) ;

		\draw  [draw opacity=0][fill={rgb, 255:red, 0; green, 0; blue, 0 }  ,fill opacity=0.25 ] (470,350) -- (530,350) -- (530,410) -- (470,410) -- cycle ;
		\draw [line width=3.75]    (470,70) -- (190,70) ;

		\draw [line width=3.75]    (530,70) -- (470,70) ;

		\draw [line width=3.75]    (470,70) -- (470,10) ;

		\draw  [draw opacity=0][fill={rgb, 255:red, 0; green, 0; blue, 0 }  ,fill opacity=0.25 ] (530,70) -- (470,70) -- (470,10) -- (530,10) -- cycle ;
		\draw [line width=3.75]    (190,70) -- (190,10) ;

		\draw [line width=3.75]    (190,70) -- (130,70) ;

		\draw  [draw opacity=0][fill={rgb, 255:red, 0; green, 0; blue, 0 }  ,fill opacity=0.25 ] (190,70) -- (130,70) -- (130,10) -- (190,10) -- cycle ;
		
		\draw [line width=3.75]    (190,350) -- (330,210) ;

		\draw [line width=3.75]    (190,70) -- (330,210) ;

		\draw [line width=3.75]    (330,210) -- (470,70) ;

		\draw [line width=3.75]    (470,350) -- (330,210) ;

		\draw  [draw opacity=0][fill={rgb, 255:red, 0; green, 0; blue, 0 }  ,fill opacity=0.25 ] (190,10) -- (470,10) -- (470,70) -- (190,70) -- cycle ;
		\draw  [draw opacity=0][fill={rgb, 255:red, 0; green, 0; blue, 0 }  ,fill opacity=0.25 ] (190,350) -- (470,350) -- (470,410) -- (190,410) -- cycle ;
		\draw  [draw opacity=0][fill={rgb, 255:red, 0; green, 0; blue, 0 }  ,fill opacity=0.25 ] (330,210) -- (470,350) -- (190,350) -- cycle ;
		\draw  [draw opacity=0][fill={rgb, 255:red, 0; green, 0; blue, 0 }  ,fill opacity=0.25 ] (330,210) -- (190,70) -- (470,70) -- cycle ;
		
		\draw (190,350) node[VertexStyle, font={\bfseries\Huge}] (0) {c} ;
		
		\draw (470,350) node[VertexStyle, font={\bfseries\Huge}] (0) {a} ;
		
		\draw (190,70) node[VertexStyle, font={\bfseries\Huge}] (0) {b} ;
		
		\draw (470,70) node[VertexStyle, font={\bfseries\Huge}] (0) {d} ;
		
		\draw (330,210) node[VertexStyle, font={\bfseries\Huge}] (0) {e} ;

		\end{tikzpicture}
		
	}

	\captionof{figure}{}\label{K6}}

	\columnbreak
	
	\center{
	\scalebox{0.35}{
		
		\begin{tikzpicture}[x=0.75pt,y=0.75pt,yscale=-1,xscale=1]
		
		
		\draw   (130,10) -- (530,10) -- (530,410) -- (130,410) -- cycle ;
		\draw  [color={rgb, 255:red, 0; green, 0; blue, 0 }  ,draw opacity=1 ][fill={rgb, 255:red, 0; green, 0; blue, 0 }  ,fill opacity=1 ] (120,230) -- (130,210) -- (140,230) -- (130,220) -- cycle ;
		\draw  [color={rgb, 255:red, 0; green, 0; blue, 0 }  ,draw opacity=1 ][fill={rgb, 255:red, 0; green, 0; blue, 0 }  ,fill opacity=1 ] (540,190) -- (530,210) -- (520,190) -- (530,200) -- cycle ;
		\draw  [color={rgb, 255:red, 0; green, 0; blue, 0 }  ,draw opacity=1 ][fill={rgb, 255:red, 0; green, 0; blue, 0 }  ,fill opacity=1 ] (310,0) -- (330,10) -- (310,20) -- (320,10) -- cycle ;
		\draw  [color={rgb, 255:red, 0; green, 0; blue, 0 }  ,draw opacity=1 ][fill={rgb, 255:red, 0; green, 0; blue, 0 }  ,fill opacity=1 ] (290,0) -- (310,10) -- (290,20) -- (300,10) -- cycle ;
		\draw  [color={rgb, 255:red, 0; green, 0; blue, 0 }  ,draw opacity=1 ][fill={rgb, 255:red, 0; green, 0; blue, 0 }  ,fill opacity=1 ] (310,400) -- (330,410) -- (310,420) -- (320,410) -- cycle ;
		\draw  [color={rgb, 255:red, 0; green, 0; blue, 0 }  ,draw opacity=1 ][fill={rgb, 255:red, 0; green, 0; blue, 0 }  ,fill opacity=1 ] (290,400) -- (310,410) -- (290,420) -- (300,410) -- cycle ;
		\draw [line width=3.75]    (130,350) -- (190,350) ;

		\draw [line width=3.75]    (190,350) -- (190,410) ;

		\draw  [draw opacity=0][fill={rgb, 255:red, 0; green, 0; blue, 0 }  ,fill opacity=0.25 ] (130,350) -- (190,350) -- (190,410) -- (130,410) -- cycle ;
		\draw [line width=3.75]    (470,350) -- (470,410) ;

		\draw [line width=3.75]    (470,350) -- (530,350) ;

		\draw  [draw opacity=0][fill={rgb, 255:red, 0; green, 0; blue, 0 }  ,fill opacity=0.25 ] (470,350) -- (530,350) -- (530,410) -- (470,410) -- cycle ;
		\draw [line width=3.75]    (470,70) -- (190,70) ;

		\draw [line width=3.75]    (530,70) -- (470,70) ;

		\draw [line width=3.75]    (470,70) -- (470,10) ;

		\draw  [draw opacity=0][fill={rgb, 255:red, 0; green, 0; blue, 0 }  ,fill opacity=0.25 ] (530,70) -- (470,70) -- (470,10) -- (530,10) -- cycle ;
		\draw [line width=3.75]    (190,70) -- (190,10) ;

		\draw [line width=3.75]    (190,70) -- (130,70) ;

		\draw  [draw opacity=0][fill={rgb, 255:red, 0; green, 0; blue, 0 }  ,fill opacity=0.25 ] (190,70) -- (130,70) -- (130,10) -- (190,10) -- cycle ;
		
		\draw [line width=3.75]    (190,350) -- (330,210) ;

		\draw [line width=3.75]    (470,350) -- (330,210) ;

		\draw  [draw opacity=0][fill={rgb, 255:red, 0; green, 0; blue, 0 }  ,fill opacity=0.25 ] (190,10) -- (470,10) -- (470,70) -- (190,70) -- cycle ;
		\draw  [draw opacity=0][fill={rgb, 255:red, 0; green, 0; blue, 0 }  ,fill opacity=0.25 ] (190,350) -- (470,350) -- (470,410) -- (190,410) -- cycle ;
		\draw  [draw opacity=0][fill={rgb, 255:red, 0; green, 0; blue, 0 }  ,fill opacity=0.25 ] (330,210) -- (470,350) -- (190,350) -- cycle ;
		\draw [line width=3.75]    (230,410) -- (330,210) ;

		\draw [line width=3.75]    (230,10) -- (190,70) ;

		\draw [line width=3.75]    (430,410) -- (330,210) ;

		\draw [line width=3.75]    (430,10) -- (470,70) ;
		
		\draw [line width=3.75]    (130,410) -- (190,350) ;

		\draw [line width=3.75]    (470,350) -- (530,410) ;

		\draw (190,350) node[VertexStyle, font={\bfseries\Huge}] (0) {c} ;
		
		\draw (470,350) node[VertexStyle, font={\bfseries\Huge}] (0) {a} ;
		
		\draw (190,70) node[VertexStyle, font={\bfseries\Huge}] (0) {b} ;
		
		\draw (470,70) node[VertexStyle, font={\bfseries\Huge}] (0) {d} ;
		
		\draw (330,210) node[VertexStyle, font={\bfseries\Huge}] (0) {e} ;

		\end{tikzpicture}

	}
	
	\captionof{figure}{}\label{K7}}

	\columnbreak
	
	\center{
	\scalebox{0.35}{
		
		\begin{tikzpicture}[x=0.75pt,y=0.75pt,yscale=-1,xscale=1]
		
%
		\draw   (130,10) -- (530,10) -- (530,410) -- (130,410) -- cycle ;
		\draw  [color={rgb, 255:red, 0; green, 0; blue, 0 }  ,draw opacity=1 ][fill={rgb, 255:red, 0; green, 0; blue, 0 }  ,fill opacity=1 ] (120,130) -- (130,110) -- (140,130) -- (130,120) -- cycle ;
		\draw  [color={rgb, 255:red, 0; green, 0; blue, 0 }  ,draw opacity=1 ][fill={rgb, 255:red, 0; green, 0; blue, 0 }  ,fill opacity=1 ] (540,290) -- (530,310) -- (520,290) -- (530,300) -- cycle ;
		\draw  [color={rgb, 255:red, 0; green, 0; blue, 0 }  ,draw opacity=1 ][fill={rgb, 255:red, 0; green, 0; blue, 0 }  ,fill opacity=1 ] (310,0) -- (330,10) -- (310,20) -- (320,10) -- cycle ;
		\draw  [color={rgb, 255:red, 0; green, 0; blue, 0 }  ,draw opacity=1 ][fill={rgb, 255:red, 0; green, 0; blue, 0 }  ,fill opacity=1 ] (290,0) -- (310,10) -- (290,20) -- (300,10) -- cycle ;
		\draw  [color={rgb, 255:red, 0; green, 0; blue, 0 }  ,draw opacity=1 ][fill={rgb, 255:red, 0; green, 0; blue, 0 }  ,fill opacity=1 ] (310,400) -- (330,410) -- (310,420) -- (320,410) -- cycle ;
		\draw  [color={rgb, 255:red, 0; green, 0; blue, 0 }  ,draw opacity=1 ][fill={rgb, 255:red, 0; green, 0; blue, 0 }  ,fill opacity=1 ] (290,400) -- (310,410) -- (290,420) -- (300,410) -- cycle ;
		\draw [line width=3.75]    (190,350) -- (470,350) ;

		\draw [line width=3.75]    (130,350) -- (190,350) ;

		\draw [line width=3.75]    (190,350) -- (190,410) ;

		\draw  [draw opacity=0][fill={rgb, 255:red, 0; green, 0; blue, 0 }  ,fill opacity=0.25 ] (130,350) -- (190,350) -- (190,410) -- (130,410) -- cycle ;
		\draw [line width=3.75]    (470,350) -- (470,410) ;

		\draw [line width=3.75]    (470,350) -- (530,350) ;

		\draw  [draw opacity=0][fill={rgb, 255:red, 0; green, 0; blue, 0 }  ,fill opacity=0.25 ] (470,350) -- (530,350) -- (530,410) -- (470,410) -- cycle ;
		\draw [line width=3.75]    (470,70) -- (190,70) ;

		\draw [line width=3.75]    (530,70) -- (470,70) ;

		\draw [line width=3.75]    (470,70) -- (470,10) ;

		\draw  [draw opacity=0][fill={rgb, 255:red, 0; green, 0; blue, 0 }  ,fill opacity=0.25 ] (530,70) -- (470,70) -- (470,10) -- (530,10) -- cycle ;
		\draw [line width=3.75]    (190,70) -- (190,10) ;

		\draw [line width=3.75]    (190,70) -- (130,70) ;

		\draw  [draw opacity=0][fill={rgb, 255:red, 0; green, 0; blue, 0 }  ,fill opacity=0.25 ] (190,70) -- (130,70) -- (130,10) -- (190,10) -- cycle ;
		
		\draw [line width=3.75]    (190,350) -- (330,210) ;

		\draw [line width=3.75]    (190,70) -- (330,210) ;

		\draw [line width=3.75]    (330,210) -- (470,70) ;

		\draw [line width=3.75]    (470,350) -- (330,210) ;

		\draw  [draw opacity=0][fill={rgb, 255:red, 0; green, 0; blue, 0 }  ,fill opacity=0.25 ] (190,10) -- (470,10) -- (470,70) -- (190,70) -- cycle ;
		\draw  [draw opacity=0][fill={rgb, 255:red, 0; green, 0; blue, 0 }  ,fill opacity=0.25 ] (190,350) -- (470,350) -- (470,410) -- (190,410) -- cycle ;
		\draw [line width=3.75]    (330,210) -- (530,210) ;

		\draw [line width=3.75]    (470,70) -- (530,210) ;

		\draw [line width=3.75]    (470,350) -- (530,210) ;

		\draw [line width=3.75]    (190,70) -- (130,210) ;

		\draw [line width=3.75]    (130,210) -- (190,350) ;

		\draw  [draw opacity=0][fill={rgb, 255:red, 0; green, 0; blue, 0 }  ,fill opacity=0.25 ] (130,70) -- (530,70) -- (530,350) -- (130,350) -- cycle ;

		\draw (190,350) node[VertexStyle, font={\bfseries\Huge}] (0) {c} ;
		
		\draw (470,350) node[VertexStyle, font={\bfseries\Huge}] (0) {a} ;
		
		\draw (190,70) node[VertexStyle, font={\bfseries\Huge}] (0) {b} ;
		
		\draw (470,70) node[VertexStyle, font={\bfseries\Huge}] (0) {d} ;
		
		\draw (330,210) node[VertexStyle, font={\bfseries\Huge}] (0) {e} ;
		
		\draw (130,210) node[VertexStyle, font={\bfseries\Huge}] (0) {f} ;
		
		\draw (530,210) node[VertexStyle, font={\bfseries\Huge}] (0) {f} ;
		
		\end{tikzpicture}

	}
	
	\captionof{figure}{}\label{K8}}

\end{multicols}

\section{Edge-maximal graphs on high genus orientable surfaces}

In this section we prove Theorems \ref{1972 impures} and \ref{Wood graphs}. To start, we prove a theorem concerning graphs whose embedding in a given surface all have some fixed faces and give an additive genus theorem on how these properties are preserved under a certain join operation on two graphs. We hope that others will also be able to make use of this theorem either in constructing other Euler impure graphs on surfaces or for entirely different purposes.

First we will have to define the join operation. We consider graphs with distinguished cycles. Given a graph $G$ and a cycle $C$ in $G$, we let $G^C$ be the graph $G$ with distinguished cycle $C$.

A graph $T$ is \emph{cylindrical}, if it is planar, has minimum degree 3 and there exists a partition $V(T)=V_1\cup V_2$ such that both $T[V_1]$ and $T[V_2]$ are induced cycles. In particular this implies that for each vertex $v$ of $V_i$, there is an edge incident with $v$ and some vertex of $V_{3-i}$. Such a graph is 3-connected and so it has a unique embedding in the plane with $V_1$ and $V_2$ being the vertices of two disjoint faces. In this way we can view $T$ as being embedded on the cylinder $\mathbb{S}_0-2\mathbb{D}$ where $T[V_1]$ lies on one boundary and $T[V_2]$ lies on the other. Given a cylindrical graph $T$, we let $T^1$ and $T^2$ denote the two facial cycles on the vertices $V_1$ and $V_2$ respectively.

Given two graphs with distinguished cycles $G^C$ and $H^D$, we call a cylindrical graph $T$, \emph{$T$-joining} if $C=T^1$ and $D=T^2$. If $T$ is $T$-joining then the \emph{cylindrical $T$-join} of $G^C$ and $H^D$, denoted by $T(G^C , H^D)$ is equal to $G\cup H \cup E(T)$. Notice that the edges $E(V(G),V(H))$ are exactly the edges of $T$ not lying on $T^1$ or $T^2$. Given a cycle $C$, let $\overline{C}$ denote the same cycle but with opposite orientation.

Let $\gamma(G)$ denote the orientable genus of a graph $G$ and $\overline{\gamma}(G)$ the non-orientable genus. Let $\gamma^*(G)=\min\{2\gamma(G), \overline{\gamma}(G)\}$ denote the Euler genus of $G$. A \emph{band} $\sigma$ of a surface $\Sigma$ is a restriction of $\Sigma$ that is homeomorphic to $\mathbb{S}_0-2\mathbb{D}$.

\begin{lemma}\label{consitant face lemma}
	Let $G^C$ and $H^D$ be two graphs with distinguished cycles such that there is some minimum orientable genus embedding of $G$ with $C$ as a facial cycle and furthermore in any minimum orientable genus embedding, the vertices of $C$ appear on a facial cycle only if the facial cycle has length $|V(C)|$, and similarly for $H^D$. Let $T$ be a cylindrical $T$-joining graph of $G^C$ and $H^D$.
	
	Then $\gamma(T(G^C , H^D))=\gamma(G)+\gamma(H)$ and furthermore in an embedding of $T(G^C , H^D)$ on $\mathbb{S}_{\gamma(G)+\gamma(H)}$, there exists a band $\sigma$ of $\mathbb{S}_{\gamma(G)+\gamma(H)}$ whose interior intersects the edges $E(V(G),V(H))$ only, one boundary intersects exactly a cycle on $V(C)$, the other boundary intersects exactly a cycle on $V(D)$ and such that $\mathbb{S}_{\gamma(G)+\gamma(H)}-\sigma \cong (\mathbb{S}_{\gamma(G)}-\mathbb{D}) \cup (\mathbb{S}_{\gamma(H)}-\mathbb{D})$.
\end{lemma}

\begin{proof}
	First we show existence of such an embedding. Consider an embedding of $G$ on $\mathbb{S}_{\gamma(G)}-\mathbb{D}$ such that $C$ lies on the boundary and similarly an embedding of $H$ on $\mathbb{S}_{\gamma(H)}-\mathbb{D}$ such that $\overline{D}$ lies on the boundary. Next consider $T$ embedded naturally on the cylinder $\sigma=\mathbb{S}_0-2\mathbb{D}$ as previously described and identify the boundary which $T^1=C$ lies on along with the boundary of $\mathbb{S}_{\gamma(G)}-\mathbb{D}$ (which $C$ also lies on) with the same orientation and identifying the cycle $T^1=C$ of $T$ with the cycle $C$ of $G$ to obtain a new graph $G^*$ embedded on this new surface which is still homeomorphic to $\mathbb{S}_{\gamma(G)}-\mathbb{D}$. Lastly identify the boundary of $\mathbb{S}_{\gamma(G)}-\mathbb{D}$ (which $T^2=D$ lies on) with the boundary of $\mathbb{S}_{\gamma(H)}-\mathbb{D}$ (which $\overline{D}$ lies on) in the opposite directions and then identify cycles $T^2=D$ of $G^*$ with cycle $D$ of $H$ to obtain $T(G^C , H^D)$ with an embedding on $\mathbb{S}_{\gamma(G)+\gamma(H)}$. Now by the construction it is clear that $\sigma$ is such a band of $\mathbb{S}_{\gamma(G)+\gamma(H)}$ in this embedding of $T(G^C , H^D)$.
	
	As the orientable genus of a graph is additive on components \cite{mohar2001graphs} and $G\cup H$ is a subgraph of $T(G^C , H^D)$, we have that $\gamma(T(G^C , H^D)) \ge \gamma(G\cup H) = \gamma(G)+\gamma(H)$. Hence $\gamma(T(G^C , H^D)) = \gamma(G)+\gamma(H)$ as $T(G^C , H^D)$ embeds on $\mathbb{S}_{\gamma(G)+\gamma(H)}$.
	
	Now fix an embedding of $T(G^C , H^D)$ on $\mathbb{S}_{\gamma(G)+\gamma(H)}$ and then consider ${\cal S}=\mathbb{S}_{\gamma(G)+\gamma(H)} - (E(G)\cup E(H))$. As $G\cup H$ has orientable genus $\gamma(G)+\gamma(H)$, we see that $\cal S$ must consist of a collection of genus 0 surfaces with boundaries. Furthermore as $G\cup H$ has two components, we see that all but one component of $\cal S$ must be homeomorphic to $\mathbb{D}$ whose boundary is associated with a closed walk of either $G$ or $H$ and one component $\sigma$ must be cylindrical, being homeomorphic to $\mathbb{S}_0-2\mathbb{D}$ with one boundary associated with a closed walk of $G$ and the other with a closed walk of $H$. In particular $\sigma$ is a band of $\mathbb{S}_{\gamma(G)+\gamma(H)}$ such that $\mathbb{S}_{\gamma(G)+\gamma(H)}-\sigma$ is disconnected.
	
	By definition, the interior of a component of $\cal S$ (in particular $\sigma$) cannot contain an edge of $G$ or $H$.
	Furthermore, each of the components of $\cal S$ that are homeomorphic to $\mathbb{D}$ is a face of the embedding of $T(G^C , H^D)$ on $\mathbb{S}_{\gamma(G)+\gamma(H)}$ as the embedding of $T(G^C , H^D)$ is a 2-cell embedding and no such component can contain an edge of $E(V(G),V(H))$ as none have a vertex of both $G$ and $H$. The interior of the other component $\sigma$ must therefore contain the edges $E(V(G),V(H))$. As $G\cup H$ embeds on $\mathbb{S}_{\gamma(G)+\gamma(H)}-\sigma$, we see that $\mathbb{S}_{\gamma(G)+\gamma(H)}-\sigma \cong (\mathbb{S}_{\gamma(G)}-\mathbb{D}) \cup (\mathbb{S}_{\gamma(H)}-\mathbb{D})$.
	
	Finally, we show that one boundary of $\sigma$ intersects exactly a cycle on $V(C)$, the other boundary intersects exactly a cycle on $V(D)$. Let $W_1$ and $W_2$ be the closed walks along each boundary of $\sigma$, with $V(W_1)\subseteq V(G)$ and $V(W_2)\subseteq V(H)$. As every vertex of $C$ is adjacent to some vertex of $D$, we see that $V(C)\subseteq W_1$. By removing $\sigma$ and identifying this new boundary on the component of $\mathbb{S}_{\gamma(G)+\gamma(H)}-\sigma$ that $G$ embeds, we obtain a minimum orientable genus embedding of $G$ with $W_1$ as a facial cycle. Hence $W_1$ has length $|V(C)|$, and so $W_1$ is a cycle on the vertices $V(C)$. Similarly for $D$ and $W_2$.
	
	So $\sigma$ is indeed the desired band completing the proof.
\end{proof}

\begin{theorem}\label{consitant face theorem}
	Let $G^C$ and $H^D$ be two graphs with distinguished cycles such that there is some minimum orientable genus embedding of $G$ with $C$ as a facial cycle and furthermore in any minimum orientable genus embedding, the vertices of $C$ appear on a facial cycle only if the facial cycle has length $|V(C)|$, and similarly for $H^D$. Let $T$ be a cylindrical $T$-joining graph of $G^C$ and $H^D$.
	
	Then $\gamma(T(G^C , H^D))=\gamma(G)+\gamma(H)$ and in every minimum orientable genus embedding of $T(G^C , H^D)$, there exists a minimum orientable genus embedding of $G$ with a facial cycle on the vertices of $C$ such that every other face of $G$ is a face of $T(G^C , H^D)$, and similarly for $H$.
\end{theorem}

\begin{proof}
	Consider a minimum orientable genus embedding of $T(G^C , H^D)$. By Lemma \ref{consitant face lemma} $\gamma(T(G^C , H^D))=\gamma(G)+\gamma(H)$ and there exists a band $\sigma$ of $\mathbb{S}_{\gamma(G)+\gamma(H)}$ whose interior intersects the edges $E(V(G),V(H))$ only and such that $\mathbb{S}_{\gamma(G)+\gamma(H)}-\sigma \cong (\mathbb{S}_{\gamma(G)}-\mathbb{D}) \cup (\mathbb{S}_{\gamma(H)}-\mathbb{D})$ and one boundary of $\sigma$ intersects exactly a cycle $W_1$ on the vertices $V(C)$ and the other boundary intersects exactly a cycle $W_2$ on $V(D)$.
	
	Now consider the sub-embedding of $G$ on $\mathbb{S}_{\gamma(G)+\gamma(H)}$, then cut along the boundary of $\sigma$ that $W_1$ lies on to obtain two components, one of which being homeomorphic to $\mathbb{S}_{\gamma(G)}-\mathbb{D}$ which $G$ is embedded on with $W_1$ lying on the boundary. Every face of this embedding of $G$ on $\mathbb{S}_{\gamma(G)}-\mathbb{D}$ is a face of the embedding of $T(G^C , H^D)$ on $\mathbb{S}_{\gamma(G)+\gamma(H)}$. By identifying a disk with the boundary of $\mathbb{S}_{\gamma(G)}-\mathbb{D}$, we obtain a minimum genus embedding of $G$ on $\mathbb{S}_{\gamma(G)}$ with just a single additional face $T(G^C , H^D)$ as required.
	
	Similarly for $H$.
\end{proof}

We remark that similar and possibly more general versions of Theorem \ref{consitant face theorem} are certainly possible. In particular an Euler genus analogue is possible with near identical arguments.

\begin{theorem}\label{Euler consitant face theorem}
	Let $G^C$ and $H^D$ be two graphs with distinguished cycles such that there is some minimum Euler genus embedding of $G$ with $C$ as a facial cycle and furthermore in any minimum Euler genus embedding, the vertices of $C$ appear on a facial cycle only if the facial cycle has length $|V(C)|$, and similarly for $H^D$. Let $T$ be a cylindrical $T$-joining graph of $G^C$ and $H^D$.
	
	Then $\gamma^*(T(G^C , H^D))=\gamma^*(G)+\gamma^*(H)$ and in every minimum Euler genus embedding of $T(G^C , H^D)$, there exists a minimum Euler genus embedding of $G$ with a facial cycle on the vertices of $C$ such that every other face of $G$ is a face of $T(G^C , H^D)$, and similarly for $H$.
	
	If at least one of $G$ or $H$ has such a minimum Euler genus embedding on a non-orientable surface, then $T(G^C , H^D)$ has such a minimum Euler genus embedding on a non-orientable surface.
\end{theorem}

However the analogous statement for non-orientable surfaces is unfortunately false. Consider for a suitable cylindrical $T$-joining graph $T$ of two copies of $K_7^{C_3}$. Then $\overline{\gamma}(K_7)=3$, but $\overline{\gamma}(T(K_7^{C_3},K_7^{C_3})=5<6$ as $\gamma(K_7)=1$ and so $T(K_7^{C_3},K_7^{C_3})$ embeds on $\mathbb{S}_2$.

Next we define a graph that will be a useful gadget in our constructions. An $n$-ladder is a $2n+2$ vertex graph $L$ with vertices $x_0,\dots ,x_n, y_0,\dots ,y_n$ such that;
\begin{itemize}
	\item Both $L[\{x_0,\dots, x_n\}]$ and $L[\{y_0,\dots, y_n\}]$ induce paths $x_0x_1\dots x_n$ and $y_0y_1\dots y_n$ respectively.
	\item The vertices $x_i$ and $y_j$ are adjacent if and only if $i=j$.
\end{itemize}

A \emph{hanging $n$-ladder} $H_n$ is an $n$-ladder with a single additional dominating vertex $h$. Note that $H_n$ is planar and 3-connected. For $i\in \{1,\dots ,n\}$, let $X_i$ denote the facial cycle $x_{i-1}y_{i-1}y_ix_i$ of $H_n$. Given a 3-connected planar graph $P$ with an outer 4-face, let $H_n(P)$ denote the graph obtained from the hanging ladder $H_n$ by identifying the facial cycle $X_n$ of $H_n$ with the outer facial 4-cycle of $P$. Notice that $H_n(P)$ is again planar and 3-connected, and so has a unique embedding in the plane.

Next we describe three more useful gadgets.

Let $T_k$ denote the cylindrical graph such that $T^1_k=a_0a_1\dots a_{k-1}$, $T^2_k=b_0b_1\dots b_{k-1}$ and $E(T_k)=E(T_k^1)\cup E(T_k^2) \cup \{a_ib_j : i-j\equiv 0 \text{ or } 1 \text{ (mod }k\text{)}\}$. Note that all faces in a planar embedding of $T_k$ are triangular except possibly $T_k^1$ and $T_k^2$.

Let $Y$ be some 4-cycle of $K_7-e$ containing the pair of none adjacent vertices.

\begin{proposition}\label{K_7}
	There exists a minimum genus embedding of $K_7-e$ with $Y$ as a facial cycle, and in all minimum genus embeddings, the vertices of $Y$ appear on a facial cycle only if the facial cycle has length 4.
\end{proposition}

\begin{proof}
Clearly $\gamma(K_7-e)=1$. Notice that as $K_7$ triangulates the torus, there exists an embedding of $K_7-e$ on the torus such that $Y$ is a facial cycle. Lastly by Euler's formula any embedding of $K_7-e$ has exactly one non-triangular face, a square face, the only one whose vertices could contain $V(Y)$.
\end{proof}

The complete graph $K_8$ has minimum orientable genus 2. By Euler's formula all minimum orientable genus embeddings in the double torus $\mathbb{S}_2$ are two edges short of a triangulation. As observed by Sun \cite{sun2017face}, every embedding of $K_8$ on the double torus $\mathbb{S}_2$ has two 4-faces. This is because if $K_8$ could embed with a 5-face, then $K_9-E(K_{1,4})$ would embed and triangulate the double torus, contradicting that there is no 9-vertex triangulation \cite{jungerman1980minimal,huneke1978minimum}.

\begin{proposition}\label{K_8}
	Every embedding of $K_8$ on the double torus $\mathbb{S}_2$ has two 4-faces.
\end{proposition}

The fact that $K_8$ has no minimum genus embedding with just a single non-triangular face will be crucial in our constructions. Sun went on to prove that with this one exception, there is a minimum (orientable or non-orientable) genus embedding of every $K_n$ with at most one non-triangular face \cite{sun2017face}. So $K_8$ having no such embedding is rather exceptional. 

We denote by $K_8^{C_4}$ the graph $K_8$ with an arbitrary distinguished cycle of length 4. We are now ready to construct the graphs of Theorems \ref{1972 impures} and \ref{Wood graphs}. The idea of the construction is as follows. We start with a hanging ladder, and glue $K_8$s to the 4-faces of the ladder via $T_4$-joins. We show that this adds 2 to the genus for every $K_8$. To achieve an odd genus, we can glue in one $K_7-e$ instead of a $K_8$. Finally, we fill one 4-face of the ladder with an arbitrary planar triangulation to make the class of so constructed graphs infinite.

For $g\ge 2$ let $F_g^0(P)=H_{\lceil \frac{g}{2}\rceil +1}(P)$. For $g\ge 2$ and $i\in \{1,\dots , \lfloor \frac{g}{2}\rfloor\}$ let $F_g^i(P)=T_4((F_g^{i-1}(P))^{X_i} , K_8^{C_4})$. For odd $g\ge 2$, let $F_g^{\lceil \frac{g}{2}\rceil}(P)= T_4((F_g^{\lfloor \frac{g}{2}\rfloor}(P))^{X_{\lceil \frac{g}{2}\rceil}} , (K_7 - e)^Y)$. Finally for $g\ge 2$ let $F_g(P)=F_g^{\lceil \frac{g}{2}\rceil}(P)$. For suitably chosen planar graphs $P$, $F_g(P)$ shall be our desired graphs.

\begin{lemma}\label{F_g even}
	Let $P$ be a 3-connected planar graph with an outer 4-face. For $i\in \{0,\dots , \lfloor \frac{g}{2}\rfloor\}$, $\gamma(F_g^i(P))=2i$ and in any minimum genus embedding of $F_g^i(P)$ there are $i$ 4-faces whose vertices induce $K_4$ and $X_{i+1},\dots , X_{\lceil \frac{g}{2}\rceil}$ are facial cycles and all other non-triangular faces are the non-triangular interior faces of $P$.
\end{lemma}

\begin{proof}
	First notice that the Lemma holds when $i=0$ as $F_g^0(P)=H_{\lceil \frac{g}{2}\rceil +1}(P)$ is planar and 3-connected, having a unique embedding with $X_{1},\dots , X_{\lceil \frac{g}{2}\rceil}$ being the only non-triangular faces other than those of $P$. We argue inductively for $i\ge 1$, so suppose that the Lemma holds for $i-1$.
	
	By definition, $F_g^i(P)=T_4((F_g^{i-1}(P))^{X_i} , K_8^{C_4})$. In any minimum genus embedding of $F_g^{i-1}(P)$, $X_i$ is a face and no other face contains all the vertices of $X_i$. By Proposition \ref{K_8}, any minimum genus embedding of $K_8$ has two 4-faces, one being $C_4$. So we may apply Theorem \ref{consitant face theorem} to $F_g^i(P)=T_4((F_g^{i-1}(P))^{X_i} , K_8^{C_4})$. First of all we have $\gamma(F_g^i(P))=\gamma(F_g^{i-1}(P)) + \gamma(K_8)=2i$.
	
	Consider a minimum genus embedding of $F_g^i(P)$. Then by Theorem \ref{consitant face theorem}, there exists a minimum genus embedding of $F_g^{i-1}(P)$ such that every face except for $X_i$ is a face of the embedding of $F_g^{i}(P)$. In particular, $X_{i+1},\dots , X_{\lceil \frac{g}{2}\rceil}$ are facial cycles, there are $i-1$ distinct 4-faces whose vertices induce $K_4$ and belong to $V(F_g^{i-1}(P))$, and non-triangular interior faces of $P$ are faces of $F_g^{i}(P)$.
	
	By Euler's formula there is just one unaccounted 4-face now. By Proposition \ref{K_8} and Theorem \ref{consitant face theorem}, the vertices of this last 4-face must be vertices of the additional $K_8$ and so therefore induce a $K_4$. Hence there are a total of $i$ distinct 4-faces whose vertices induce $K_4$ as required.
\end{proof}

\begin{lemma}\label{F_g}
	Let $P$ be a 3-connected planar graph with an outer 4-face. Then  $\gamma(F_g(P))=g$ and in any minimum genus embedding of $F_g(P)$ there are $\lfloor \frac{g}{2}\rfloor$  4-faces whose vertices induce $K_4$ and all other non-triangular faces are interior faces of $P$.
\end{lemma}

\begin{proof}
	When $g$ is even, the statement follows immediately from Lemma \ref{F_g even}. So we may assume that $g$ is odd.
	
	Note that by definition, $F_g(P)=T_4((F_g^{\lfloor \frac{g}{2}\rfloor}(P))^{X_{\lceil \frac{g}{2}\rceil}} , (K_7 - e)^Y)$. By Lemma \ref{F_g even} any embedding of $F_g^{\lfloor \frac{g}{2}\rfloor}(P)$ has $\lfloor \frac{g}{2}\rfloor$ 4-faces whose vertices induce $K_4$, $X_{\lceil \frac{g}{2}\rceil}$ is a face and all other non-triangular faces are interior faces of $P$. So by Proposition \ref{K_7}, we may apply Theorem \ref{consitant face theorem} to $F_g(P)=T_4((F_g^{\lfloor \frac{g}{2}\rfloor}(P))^{X_{\lceil \frac{g}{2}\rceil}} , (K_7 - e)^Y)$. First of all, $\gamma(F_g(P))=\gamma(F_g^{\lfloor \frac{g}{2}\rfloor}(P))+\gamma(K_7)=g$.
	
	Secondly by Theorem \ref{consitant face theorem}, $F_g(P)$ has $\lfloor \frac{g}{2}\rfloor$ 4-faces whose vertices induce $K_4$ and non-triangular interior faces of $P$ are faces of $F_g(P)$. Lastly by Euler's formula, there are no more non-triangular faces.
\end{proof}

With this Lemma, Theorems \ref{1972 impures} and \ref{Wood graphs} are now straight forward.

\begin{proof}[Proof of Theorem \ref{1972 impures}]
	Let $P$ be any 3-connected planar graph with just one non-triangular face, its outer 4-face. By Lemma \ref{F_g}, $F_g(P)$ embeds on $\mathbb{S}_g$, is $\lfloor \frac{g}{2} \rfloor$ edges short of a triangulation and in any given embedding all $\lfloor \frac{g}{2} \rfloor$ non-triangular faces are 4-faces inducing $K_4$.
	
	Hence $F_g(P)$ is edge-maximal and so is $\lfloor \frac{g}{2} \rfloor$-Euler impure. The result follows as there are infinitely many such choices of $P$.
\end{proof}

\begin{proof}[Proof of Theorem \ref{Wood graphs}]
	Let $P$ be any 3-connected planar graph which is at least $n$ edges short of triangulating the plane and has an outer 4-face. By Lemma \ref{F_g}, $F_g(P)$ embeds on $\mathbb{S}_g$, is at least $n$ edges short of a triangulation and in any given embedding there are $\lfloor \frac{g}{2} \rfloor$ 4-faces inducing $K_4$.
	
	Let $G$ be a super-graph of $F_g(P)$ on the same vertex set, then $G$ must also have $\lfloor \frac{g}{2} \rfloor$ 4-faces inducing $K_4$ and so is at least $\lfloor \frac{g}{2} \rfloor$ edges short of triangulating $\mathbb{S}_g$.
	
	The result follows as there are infinitely many such choices of $P$.
\end{proof}

With a modification of the construction in Theorem \ref{1972 impures}, Thomassen noted that for $g\ge 3$ there are infinitely many Euler impure graphs on $\mathbb{S}_g$ that are 6-connected, have clique number 6 and  are $\lfloor \frac{g-1}{2} \rfloor$ edges short of a triangulation \cite{test1}. This answers a question that was to appear in next the section.

\section{Further work}

We finish by discussing possible further problems concerning edge-maximal graphs on surfaces.

\begin{enumerate}
	\item Does there exists $\Omega(g)$-Euler impure graphs on non-orientable surfaces of genus $g$? Together with Theorem \ref{1972 impures} and McDiarmid and Wood's upper bound \cite{mcdiarmid2018edge}, a positive answer to this would provide a complete asymptotic answer to Kainen's question on how many edges short of a triangulation can a graph on a given surface $\Sigma$ be \cite{kainen1974some}.
	
	\item Which of the remaining non-orientable surfaces have just finitely many Euler impure graphs? In this direction we conjecture such a characterization for Dyck's surface $\mathbb{N}_3$ which we believe to be tractable.
	
	\begin{conjecture}
		The graphs $K_8-E(2K_2)$ and $K_8-E(K_{1,2})$ are the two unique Euler impure graphs on Dyck's surface $\mathbb{N}_3$.
	\end{conjecture}

	\item What are the Euler impure graphs on the double torus $\mathbb{S}_2$? Despite there possibly being infinite families of Euler impure graphs on many surfaces, there may well still exist reasonable characterizations. Of particular interest is the double torus $\mathbb{S}_2$. We also believe such a characterization for $\mathbb{S}_2$ to be tractable but make no conjecture of what this characterization may be.
	
	
	\item Is there an Euler impure graph with clique number 4?
	Currently just one Euler impure graph with clique number 5 is known, $K_8-E(C_5)$ on the torus.
	
	\item Does there exists an Euler impure graph $G$ on some surface $\Sigma$, such that $G$ admits an embedding into $\Sigma$, with a face of size at least 5? What about a face of size at least $n$ for any given $n\ge 5$? All known Euler impure graphs embed with just triangular and square faces. It would be interesting to determine whether or not this is necessary.
	
	\item Does there exist a graph which is Euler impure on two distinct surfaces? Such a graph or family providing a unified answer to any of these questions on both orientable and non-orientable surfaces would be particularly nice.
	
	\item What about bipartite graphs? Similar questions can be asked for graphs that are edge-maximal with respect to being embeddable in a given surface and being bipartite. Similarity in this case the graph could be a complete bipartite graph, it could quadrangulate the surface or otherwise it would non-complete and not quadrangulate the surface. No such edge-maximal bipartite graph on a surface is known. It would be interesting to study these with all the same questions as before.
\end{enumerate}

A difficulty in studying Euler impure graphs on surfaces is having only a few examples to examine and understand. Despite constructing an infinite family on orientable surfaces it is still important to find more examples, either sporadic or more infinite families.

\section*{Acknowledgments}
We learned about this problem after it was presented by Joseph Doolittle at the Graduate Research Workshop in Combinatorics in Ames in 2018. We would also like to thank Carsten Thomassen for some insightful remarks.

\bibliographystyle{amsplain}
\bibliography{Impuregraphs.bib}

\providecommand{\bysame}{\leavevmode\hbox to3em{\hrulefill}\thinspace}
\providecommand{\MR}{\relax\ifhmode\unskip\space\fi MR }
\providecommand{\MRhref}[2]{%
  \href{http://www.ams.org/mathscinet-getitem?mr=#1}{#2}
}
\providecommand{\href}[2]{#2}
\begin{thebibliography}{10}

\bibitem{Davies2019}
James Davies, \emph{On projective-planar graphs with minimum degree 5}, In
  Preparation (2019).

\bibitem{dehkordi2019non}
Hooman~R Dehkordi and Graham Farr, \emph{Non-separating planar graphs}, arXiv
  preprint arXiv:1907.09817 (2019).

\bibitem{duke1972genus}
RA~Duke and G~Haggard, \emph{The genus of subgraphs of $K_8$}, Israel Journal of
  Mathematics \textbf{11} (1972), no.~4, 452--455.

\bibitem{franklin1934six}
Philip Franklin, \emph{A six color problem}, Journal of Mathematics and Physics
  \textbf{13} (1934), no.~1-4, 363--369.

\bibitem{harary1974maximal}
Frank Harary, Paul~C Kainen, Allen~J Schwenk, and Arthur~T White, \emph{A
  maximal toroidal graph which is not a triangulation}, Mathematica
  Scandinavica (1974), 108--112.

\bibitem{huneke1978minimum}
John~Philip Huneke, \emph{A minimum-vertex triangulation}, Journal of
  Combinatorial Theory, Series B \textbf{24} (1978), no.~3, 258--266.

\bibitem{jackson2000atlas}
David Jackson and Terry~I Visentin, \emph{An atlas of the smaller maps in
  orientable and nonorientable surfaces}, Chapman and Hall/CRC, 2000.

\bibitem{jungerman1980minimal}
Mark Jungerman and Gerhard Ringel, \emph{Minimal triangulations on orientable
  surfaces}, Acta Mathematica \textbf{145} (1980), no.~1, 121--154.

\bibitem{kainen1974some}
Paul~C Kainen, \emph{Some recent results in topological graph theory}, Graphs
  and combinatorics, Springer, 1974, pp.~76--108.

\bibitem{mader1968homomorphiesatze}
Wolfgang Mader, \emph{Homomorphies{\"a}tze f{\"u}r graphen}, Mathematische
  Annalen \textbf{178} (1968), no.~2, 154--168.

\bibitem{mcdiarmid2019purity}
Colin McDiarmid and Micha{\l} Przykucki, \emph{On the purity of minor-closed
  classes of graphs}, Journal of Combinatorial Theory, Series B \textbf{135}
  (2019), 295--318.

\bibitem{mcdiarmid2018edge}
Colin McDiarmid and David~R Wood, \emph{Edge-maximal graphs on surfaces},
  Canadian Journal of Mathematics \textbf{70} (2018), no.~4, 925--942.

\bibitem{mohar2001graphs}
Bojan Mohar and Carsten Thomassen, \emph{Graphs on surfaces}, vol.~10, JHU
  Press, 2001.

\bibitem{ringel1955man}
Gerhard Ringel, \emph{Wie man die geschlossenen nichtorientierbaren fl{\"a}chen
  in m{\"o}glichst wenig dreiecke zerlegen kann}, Mathematische Annalen
  \textbf{130} (1955), no.~4, 317--326.

\bibitem{sachs1983spatial}
Horst Sachs, \emph{On a spatial analogue of kuratowski's theorem on planar
  graphs—an open problem}, Graph theory, Springer, 1983, pp.~230--241.

\bibitem{sun2017face}
Timothy Sun, \emph{Face distributions of embeddings of complete graphs}, arXiv
  preprint arXiv:1708.02092 (2017).

\bibitem{test1}
Carsten Thomassen, \emph{Personal communication}, 2019.

\end{thebibliography}

\end{document}